\documentclass{amsart}
\usepackage{a4wide,mathtools,bbm,amsthm,amsmath,amsfonts,amssymb,color,epic}
\usepackage{multirow,comment}
\usepackage[export]{adjustbox}
\usepackage{tikz,pgfplots,pgfplotstable}
\usetikzlibrary{arrows,snakes,backgrounds,positioning}

\theoremstyle{plain}             %

\newtheorem{theorem}{Theorem}[section]
\newtheorem{lemma}[theorem]{Lemma}

\newtheorem{remark}[theorem]{Remark}
\newtheorem{assumption}[theorem]{Assumption}

\renewcommand{\d}{\operatorname{d\!}}
\newcommand{\E}{\mathbb{E}}

\newcommand{\isdef}{\mathrel{\mathrel{\mathop:}=}}

\newcommand{\dist}{\operatorname{dist}}

\newcommand{\bgamma}{{\boldsymbol{\gamma}}}
\newcommand{\balpha}{{\boldsymbol{\alpha}}}
\newcommand{\bbeta}{{\boldsymbol{\beta}}}

\newcommand{\btau}{{\boldsymbol{\tau}}}

\newcommand{\bkappa}{{\boldsymbol{\kappa}}}

\newcommand{\bphi}{{\boldsymbol{\varphi}}}

\DeclareMathOperator*{\essinf}{\operatorname{ess\,inf}}
\DeclareMathOperator*{\esssup}{\operatorname{ess\,sup}}

\newcommand{\refd}{{\operatorname*{ref}}}
\newcommand{\intI}{{\operatorname*{I}}}
\newcommand{\intQ}{{\operatorname*{Q}}}

\definecolor{navy}{RGB}{102,153,255}
\definecolor{tuerkis}{rgb}{0.0, 0.5, 1.0}
\definecolor{mint}{rgb}{0.71,0.84,0.82}
\definecolor{dunkelmint}{RGB}{0,117,108}
\definecolor{Myblue}{rgb}{0.2,0.2,0.7}
\definecolor{Mymagenta}{rgb}{1,0,1}
\definecolor{Mycyan}{rgb}{0,1,1}
\definecolor{Mygreen}{rgb}{0,0.6,0}
\definecolor{awesome}{rgb}{1.0, 0.13, 0.32}
\definecolor{azure}{rgb}{0.0, 0.5, 1.0}
\definecolor{deepfuchsia}{rgb}{0.76, 0.33, 0.76}
\title[Novel results for the anisotropic sparse grid quadrature]
{Novel results for the anisotropic sparse grid quadrature}
\author{A.-L.~Haji-Ali\and H.~Harbrecht\and M.~D.~Peters\and M.~Siebenmorgen}
\address{
Abdul-Lateef Haji-Ali,
Mathematical Institute, Radcliffe Observatory Quarter, Oxford OX1 6GG, United Kingdom}
\email{abdullateef.hajiali@maths.ox.ac.uk}
\address{
Helmut Harbrecht and
Michael D.\ Peters,
Departement Mathematik und Informatik,
Universit\"at Basel,
Basel, Switzerland
}
\email{\{helmut.harbrecht,michael.peters\}@unibas.ch}

\address{
Markus Siebenmorgen,
Institut f\"ur Numerische Simulation, Universit\"at Bonn, Bonn, Germany}
\email{siebenmorgen@ins.uni-bonn.de}
\subjclass[2000]{65D30, 65C30, 60H25}

\begin{document}
\begin{abstract}
This article is dedicated to the anisotropic sparse grid quadrature for
functions which are analytically extendable into an anisotropic tensor product
domain. Taking into account this anisotropy, we end up with a dimension independent 
error versus cost estimate of the proposed quadrature. 
In addition, we provide a novel and improved estimate for the cardinality of the underlying
anisotropic index set. To validate the theoretical findings, we present
several examples ranging from simple quadrature problems to diffusion problems
on random domains. These examples demonstrate the remarkable convergence 
behaviour of the anisotropic sparse grid quadrature in applications.
\end{abstract}
\maketitle
\section{Introduction}
This article is dedicated to the construction of anisotropic 
sparse grid quadrature methods for functions which are 
analytically extendable into an anisotropic tensor product domain.
Specifically, we will develop and analyze a particular realization 
based on Gauss-Legendre quadrature rules. Anisotropic sparse 
grid quadrature methods can be seen as a generalization of 
sparse Smolyak type quadratures, cf.\ \cite{Smo63}, since they 
are explicitly tailored to the anisotropic behaviour of the underlying 
integrand. Taking into account these anisotropies leads to a remarkable 
improvement in the cost of the sparse grid quadrature. 

Usually, a sparse grid quadrature is described by some sparse index set 
and a sequence of univariate quadrature rules. For the sequence of univariate 
quadrature rules, we employ Gauss-Legendre quadratures with linearly 
increasing numbers of quadrature points. The index set is a priorily 
chosen with respect to a certain weight vector, which incorporates 
the anisotropy, and a predefined approximation level. It is also 
possible to adaptively select those indices which provide the 
main contribution to the integral, see \cite{GG03}. Such adaptive 
methods have successfully been applied in the context of random 
diffusion problems, see e.g.~\cite{CCS14b,NTTT16,SS13}, in order to 
compute best \(N\)-term approximations of the corresponding solution. 
However, the adaptive construction of index sets is computational 
expensive and only heuristic error indicators are available. 
Hence, it can not be guaranteed that the adaptively selected 
index set is optimal. 
Instead of choosing Gauss-Legendre points, a sequence of nested 
quadrature rules such as Clenshaw-Curtis or Leja type quadratures 
could be considered as well. 
While the number of quadrature points needs to be doubled for 
Clenshaw-Curtis quadratures in order to guarantee their nestedness, 
only one additional quadrature point is added for each consecutive member of the 
Leja sequence. 
Hence, based on the Leja sequences, a sparse grid quadrature can be 
constructed where only one additional function evaluation is required 
for each new multi-index in the sparse index set, cf.~\cite{GO16}. 
However, in contrast to Gaussian quadratures, the quadrature weights 
of the Leja sequence are not necessarily positive which yields that the 
stability constant of the quadrature might not be uniformly bounded. 
Moreover, Gaussian quadrature rules provide a much higher degree 
of polynomial exactness than Leja quadratures with the same number 
of quadrature points, which is particularly advantageous for smooth integrands.  

The main task in estimating
the quadrature's cost is the estimation of the number of multi-indices
which are contained in the sparse index set. For the isotropic variant, 
the number of indices can easily be determined by combinatorial arguments,
see e.g.~\cite{GG98,NR96,WW95}. Things get more involved if one considers
anisotropic, i.e.\ weighted, sparse index sets, which yields a particular 
instance of a weighted tensor product algorithm, see \cite{WW99} and 
especially \cite[Chapter 15]{NW10}, where a comprehensive overview 
of related literature can be found. In this case, to the best of our
knowledge, only very rough estimates on the cardinality of the index set
are known, although several estimates can be found in the literature, see
e.g.~\cite{Beg72}. In fact, this problem is equivalent to the estimation 
of the number of integer solutions to linear Diophantine inequalities,
see \cite{SCH} and the references therein, which is a problem in number 
theory, or to the calculation of the integer points in a convex polyhedron.
Current estimates are not sharp and do not provide improved results 
for the cost of the anisotropic sparse grid quadrature in comparison with 
the anisotropic full tensor product quadrature. In this article, we prove 
a novel formula to estimate the cardinality of the sparse index set in 
the weighted case. This formula is much more accurate than the other 
established estimates.

A very popular application that requires efficient high-dimensional
quadrature rules are parametric partial differential equations. They
are obtained, for example, from partial differential equations with random 
data by truncating the series expansions of the underlying random 
fields and parametrizing them with respect to the 
random fields' distribution. As representatives for such problems, 
we shall consider here elliptic diffusion problems with random coefficients 
as well as diffusion problems on random domains
as specific examples to quantify the performance of the anisotropic
sparse grid quadrature. The resulting quadrature approach is very similar 
to the an\-iso\-tro\-pic sparse grid collocation method based on Gaussian
collocation points which has been introduced in \cite{NTW08a,NTW08b}.
The  collocation method interpolates the random solution in certain collocation
points and represents it in the parameter space with the aid of
polynomials. Thus, it belongs to the class of non-intrusive methods,
cf.~\cite{BNT07}. Instead of representing the random solution itself, the
anisotropic sparse grid quadrature can be employed to directly compute
the solutions statistics, i.e.\ its moments, and functionals of the solution.

The remainder of this article is organized as follows.
Section~\ref{sec:SPDE} specifies the quadrature problem
under consideration and provides the corresponding framework.
The subsequent Section~\ref{sec:anisogauss} is dedicated
to the construction of the anisotropic sparse grid quadrature 
method. In particular, the main ingredients, i.e.\ the index set 
and the sequence of univariate quadrature rules are introduced.
In Section \ref{sec:errana}, we provide corresponding error estimates
with respect to the level of the anisotropic sparse grid quadrature and 
with respect to the cardinality of the index set based on the one 
dimensional error estimate for the Gauss-Legendre quadrature.
Section~\ref{sec:cost} deals with the cost of the anisotropic 
sparse grid quadrature. In particular, we state here a
novel estimate on the number of indices in the weighted sparse
tensor product and provide a proof of this estimate. In
Section~\ref{sec:NumRes}, we consider three different applications:
A pure high dimensional quadrature problem, a diffusion problem 
with random coefficient and a diffusion problem on a random domain.
Finally, we state concluding remarks in Section~\ref{sec:conclusion}.

\section{Problem setting}\label{sec:SPDE}
In what follows, let \(\Gamma\isdef[-1,1]\).
The \(\sigma\)-algebra of Borel sets on \(\Gamma\) shall be denoted by \(\mathcal{B}\) 
and \(\nu\isdef\d y/2\) is the normalized Lebesgue measure on \(\mathcal{B}\). 
We define the product probability space \((\Gamma^\infty,\mathcal{B}^\infty,\mu)\) 
of all sequences\footnote{We make the convention \(\mathbb{N}\isdef\{0,1,2,\ldots\}\) 
and \(\mathbb{N}^*\isdef\{1,2,\ldots\}.\)} \(\boldsymbol\psi\colon\mathbb{N}^*\to\Gamma\), 
where \(\boldsymbol\psi=\{\psi_n\}_n\). 
Herein,  \(\mathcal{B}^\infty\) is the \(\sigma\)-algebra generated by the cylindrical 
sets and \(\mu\) is the corresponding product measure, i.e.\ \(\mu(A_1\times\dots\times A_m\times\Gamma\times\dots)=\prod_{n=1}^m\nu(A_i)\) for
all \(m\in\mathbb{N}\) and \(A_1,\ldots,A_m\in\mathcal{B}\).
For an integrable function \(f\colon \Gamma^\infty\to\mathbb{R}\),
we are interested in the efficient approximation of the integral
\begin{equation}\label{eq:Int}
\int_{\Gamma^\infty} f\d\mu.
\end{equation}
The approach we shall present here is based on the construction of efficient quadrature formulas for
a certain surrogate \(f_m\colon\Gamma^m\to\mathbb{R}\) of
\(f\). In order to define this surrogate, we make the following assumption.
\begin{assumption}\label{ass:AnaExtension}
Let \(\Sigma_n=\Sigma(\Gamma,\tau_n)\isdef\{z\in\mathbb{C}: \dist(z,\Gamma)\leq\tau_n\}\). We assume that
\(f\) is analytically extendable into \(\boldsymbol\Sigma(\boldsymbol\tau)\isdef
\bigtimes_{n=1}^{\infty}\Sigma_n\) for an isotone sequence
\(\tau_n\to\infty\), which measures the anisotropy
of the function \(f\) with respect to the different dimensions. 
\end{assumption}

Now, given an anchor point \(\overline{\boldsymbol\psi}\in\Gamma^{\infty}\), we define the surrogate \(f_m\) of \(f\) as
the projection of \(f\) onto the first \(m\) variables anchored at 
\(\overline{\boldsymbol\psi}\), i.e.\
\[
f_m(y_1,\ldots,y_m) \isdef f(y_1,\ldots,y_m,\overline{\psi}_{m+1},\overline{\psi}_{m+2},\ldots). 
\]
The projection \(f_m\) is well-defined since \(f\) is analytic. Moreover, we 
know from Assumption\ \ref{ass:AnaExtension} that \(f_m\) is analytically 
extendable into 
\[
\boldsymbol\Sigma_m\isdef\bigtimes_{n=1}^{m}\Sigma_n
\]
and it holds, due to the Kolomogorov extension theorem, 
cf.\ \cite{Kol33}, that 
\begin{equation}\label{eq:IntApprox}
\bigg|\int_{\Gamma^\infty}f\d\mu-\int_{\Gamma^m}f_m({\bf y})2^{-m}\d{\bf y}\bigg|\leq\varepsilon(m)
\end{equation} 
with a null sequence \(\varepsilon(m)\).%

In the sequel, we shall approximate
\begin{equation}\label{eq:IntFinite}
{\bf I} f_m\isdef \bigg(\bigotimes_{n=1}^m\operatorname{I}^{(n)}\bigg) 
f_m\isdef \int_{\Gamma^m}f_m({\bf y})2^{-m}\d{\bf y}
\end{equation}
by the anisotropic sparse tensor product quadrature. Herein, we set
\[
\big(\intI^{(n)} f_m\big)({\bf y}_n^\star) \isdef \int_{\Gamma}  f_m(y_n,{\bf y}_n^\star)\frac{\d y_n}{2},
\] 
where \({\bf y}_n^\star\isdef[y_1,\ldots,y_{n-1},y_{n+1},\ldots,y_m]\) with the notational 
convenience \({\bf y} = (y_n, {\bf y}_n^\star)\). 
It is evident that the precision of the applied quadrature
has to increase when \(m\) increases. Moreover, the cost
usually scales exponentially with respect to \(m\), which is referred to 
as the ``curse of dimensionality''. Therefore, we have to keep 
track of the impact of the dimension \(m\) on the error estimates. 

We start by considering the univariate quadrature with respect to \(y_n\), \(n=1,\ldots,m\). To that end,
we introduce the \(N\)-point Gauss-Legendre quadrature operator \(\intQ^{(n)}\) 
with quadrature points \(\xi_k\in\Gamma\) and weights \(\omega_k\) according to
\[
\big(\intQ^{(n)} f_m\big)({\bf y}_n^\star) \isdef \sum_{k=1}^N \omega_k 
f_m(\xi_k,{\bf y}_n^\star).
\]
This quadrature operator is exact for all polynomials \(p\in\Pi_{2 N-1}\isdef\operatorname{span}\{1,x,\ldots,x^{2N-1}\}\). Moreover, \(\intQ^{(n)}\colon C(\Gamma)\to \mathbb{R}\) is 
continuous with continuity constant \(1\). Hence, the quadrature error can be bounded 
by 
\begin{equation}\label{eq:polappr}
\begin{aligned}
\big|(\intI^{(n)}-\intQ^{(n)}) f_m({\bf y}_n^\star)\big|&\le \inf_{p\in\Pi_{2N-1}}\big(\big|\intI^{(n)}(f_m({\bf y}_n^\star)-p)\big|
+\big|\intQ^{(n)}(f_m({\bf y}_n^\star)-p)\big|\big)\\
&\le 2 \inf_{p\in\Pi_{2N-1}}\big\|f_m({\bf y}_n^\star)-p\big\|_{C(\Gamma)}.
\end{aligned}
\end{equation}
Thus, in accordance with 
e.g.\ \cite[Chapter 7.8]{DL93} and \cite{BNT07}, the analytic extendability of \(f_m\) guarantees that the 
\(N\)-point Gauss-Legendre quadrature satisfies the one-dimensional error estimate 
\begin{equation}\label{eq:assgaussquad}
\big|(\intI^{(n)}-\intQ^{(n)}) f_m({\bf y}_n^\star)\big|\leq c(\kappa_n) \exp\big(-\log(\kappa_n) (2N-1)\big)\|f_m({\bf y}_n^\star)\|_{C(\Sigma_n)}
\end{equation}
with 
\[
\kappa_n = \tau_n+\sqrt{1+\tau_n^2}\quad\text{and}\quad c(\kappa_n) = \frac{4}{\kappa_n-1}.
\]
Here and in the sequel, we set
\[
\|f_m({\bf y}_n^\star)\|_{C(\Sigma_n)}\isdef\max_{z\in\Sigma_n}|f_m(z,{\bf y}_n^\star)|.
\]
Note that it holds by definition 
\[
\kappa_n>1+\tau_n,\quad \kappa_n>2\tau_n\quad\text{and, thus,}\quad c(\kappa_n)<\frac{4}{\tau_n}.
\] 

\section{Anisotropic sparse grid quadrature}\label{sec:anisogauss}
We shall now introduce anisotropic sparse grid quadrature
formulas which extend the original idea of Smolyak's
construction from \cite{Smo63}. The subsequent realization 
is a particular instance of a weighted tensor product algorithm 
as introduced in \cite{WW99}. Hence, all the results for those 
algorithms are valid, see \cite[Chapter 15]{NW10} for an 
overview. Compared with the general setup considered therein, 
our analysis is taylored for Gauss-Legendre quadrature rules, 
which are especially non-nested, and the specific function class 
under consideration.

We start by considering a sequence of univariate quadratures 
\(\intQ_j\) of increasing accuracy with $N_j\in\mathbb{N}^*$
points and weights
\begin{equation}\label{eq:interpp}
\theta_j\isdef\big\{\xi_{i,j}\big\}_{i=1}^{N_j},\quad \big\{\omega_{i,j}\big\}_{i=1}^{N_j},\quad j=0,1,2,\ldots,
\end{equation}
where \(N_0\le N_1 \le \cdots\) and \(N_j\to \infty\) for \(j\to \infty\). 

Since we deal with multidimensional quadrature rules, we furthermore 
define tensor products of univariate quadrature rules by 
\[
\intQ_{j_1}^{(1)}\otimes\cdots\otimes \intQ_{j_m}^{(m)} f_m \isdef 
\sum_{k_1=1}^{N_{j_1}}\ldots\sum_{k_m=1}^{N_{j_m}} 
\bigg(\prod_{i=1}^m \omega_{k_i,j_i}\bigg)
f_m(\xi_{k_1,j_1},\ldots, \xi_{k_m,j_m}).
\]

Following the notation of \cite{NR96}, we introduce for
\(j\in\mathbb{N}\) the difference quadrature operator
\begin{equation}\label{eq:deltaop}
\Delta_j\isdef \intQ_{j}-\intQ_{j-1},\quad\text{where}\quad
\intQ_{-1}\isdef 0.
\end{equation}
With the telescoping sum \(\intQ_{j}=\sum_{\ell=0}^j \Delta_\ell\),
the \emph{isotropic $m$-fold tensor product quadrature}, 
which uses \(N_j\) quadrature points in each direction, 
can be written by
\begin{equation}\label{eq:isoquadop}
\intQ_{j}^{(1)}\otimes\cdots\otimes \intQ_{j}^{(m)}
= \sum_{\|\balpha\|_{\infty}\le j }
\Delta_{\alpha_1}^{(1)}\otimes\cdots\otimes\Delta_{
\alpha_m}^{(m)},
\end{equation}
where the superscript index indicates the particular dimension. 
Since the tensor product quadrature is convergent, we observe that
\begin{equation}\label{eq:exprint}
{\bf I} (f_m) = \sum_{\balpha \in \mathbb{N}^m} \Delta_{\alpha_1}^{(1)}\otimes\cdots\otimes
\Delta_{\alpha_m}^{(m)} f_m. 
\end{equation}
The cost of applying a quadrature formula is measured by
the number of quadrature points. In case of the isotropic 
tensor product quadrature operator \eqref{eq:isoquadop} this 
number is given by \(N^m\), where \(N=N_{j_1}=\ldots=N_{j_m}\). 
Thus, this isotropic tensor product quadrature suffers extremely
from the curse of dimensionality. The classical \emph{sparse
grid quadrature}, cf.~\cite{BG04,GG98,WW95}, can overcome
this problem up to a certain extent. It is based on linear
combinations of tensor product quadrature formulas of
relatively small size. To define the sparse
quadrature, we introduce as in \cite{BNR00,NTW08a,WW95} for
each \emph{approximation level} \(q\) the sets of multi-indices
\[
X(q,m)\isdef\Big\{{\bf 0}\le\balpha\in\mathbb{N}^m:
\|\balpha\|_1\le q\Big\}
\]
and
\[
Y(q,m)\isdef\Big\{{\bf 0}\le\balpha\in\mathbb{N}^m:
q-m<\|\balpha\|_1 \le q\Big\}.
\]
The sparse grid quadrature operator, cf.~\cite{BNR00,GG98,Smo63}, 
is then given by
\begin{equation}\label{eq:smolop}
\mathcal{A}(q,m)\isdef\sum_{\balpha\in X(q,m)}
\Delta_{\alpha_1}^{(1)}\otimes\cdots\otimes\Delta_{\alpha_m}^{(m)}.
\end{equation}
An equivalent expression is obtained by the 
\emph{combination technique}, cf.\ \cite{GSZ92},
\begin{equation}\label{eq:smolop2}
\mathcal{A}(q,m)=\sum_{\balpha\in Y(q,m)}(-1)^{q-\|\balpha\|_1}
\binom{m-1}{q-\|\balpha\|_1}
{\bf Q}_{\balpha},
\quad\text{where}\quad
{\bf Q}_{\balpha}\isdef \intQ_{\alpha_1}^{(1)}\otimes\cdots\otimes
\intQ_{\alpha_m}^{(m)}.
\end{equation}
A visualization of the set of indices \(X(q,m)\) is given in Figure \ref{fig:indicessparse}.

\begin{figure}
\begin{center}
\begin{tikzpicture}[scale=.52,every node/.style={minimum size=1cm},on grid]
 \fill[white,fill opacity=0.9] (0,0) rectangle (6,6);
        \draw[step=1cm, black] (0,0) grid (6,6); %
        \draw[black,very thick] (0,0) rectangle (6,6);%
        \fill[tuerkis] (0.1,0.1) rectangle (0.9,0.9);
       	\fill[tuerkis] (1.1,0.1) rectangle (1.9,0.9);
	\fill[tuerkis] (2.1,0.1) rectangle (2.9,0.9);
	\fill[tuerkis] (3.1,0.1) rectangle (3.9,0.9);
	\fill[tuerkis] (4.1,0.1) rectangle (4.9,0.9);
	\fill[tuerkis] (5.1,0.1) rectangle (5.9,0.9);
        \fill[tuerkis] (0.1,1.1) rectangle (0.9,1.9);
        \fill[tuerkis] (1.1,1.1) rectangle (1.9,1.9);
        \fill[tuerkis] (2.1,1.1) rectangle (2.9,1.9);
        \fill[tuerkis] (3.1,1.1) rectangle (3.9,1.9);
	\fill[tuerkis] (4.1,1.1) rectangle (4.9,1.9);
	\fill[tuerkis] (0.1,2.1) rectangle (0.9,2.9);
        \fill[tuerkis] (1.1,2.1) rectangle (1.9,2.9);
        \fill[tuerkis] (2.1,2.1) rectangle (2.9,2.9);
        \fill[tuerkis] (3.1,2.1) rectangle (3.9,2.9);
	\fill[tuerkis] (0.1,3.1) rectangle (0.9,3.9);
        \fill[tuerkis] (1.1,3.1) rectangle (1.9,3.9);
        \fill[tuerkis] (2.1,3.1) rectangle (2.9,3.9);
        \fill[tuerkis] (0.1,4.1) rectangle (0.9,4.9);
        \fill[tuerkis] (1.1,4.1) rectangle (1.9,4.9);
        \fill[tuerkis] (0.1,5.1) rectangle (0.9,5.9);

 \draw[->,thick] (0,0) -- (6.5,0) node[right] {$\alpha_1$};
    	\draw[->,thick] (0,0) -- (0,6.5) node[above] {$\alpha_2$};
    	\foreach \x/\xtext in {0/0, 1/1, 2/2, 3/3, 4/4, 5/5}
    	\draw[shift={(\x+0.5,0)}] (0pt,2pt) -- (0pt,-2pt) node[below] {$\xtext$};
    	\foreach \y/\ytext in {0/0, 1/1, 2/2, 3/3, 4/4, 5/5}
    	\draw[shift={(0,\y+0.5)}] (2pt,0pt) -- (-2pt,0pt) node[left] {$\ytext$};

\end{tikzpicture}
\hspace*{-1cm}
\begin{tikzpicture}[scale=.4,every node/.style={minimum size=1cm},on grid]
     \begin{scope}[
            yshift=-160,every node/.append style={
            yslant=0.5,xslant=-1},yslant=0.5,xslant=-1
            ]
                  \fill[white,fill opacity=.9] (0,0) rectangle (6,6);
        \draw[black,very thick] (0,0) rectangle (6,6);
        \draw[step=1cm, black] (0,0) grid (6,6);
       \fill[tuerkis] (0.1,0.1) rectangle (0.9,0.9);

        \draw[->,thick] (0,0) -- (6.5,0) node[right] {$\mathsf{\alpha_1}$};
    	\draw[->,thick] (0,0) -- (0,6.5) node[above] {$\mathsf{\alpha_2}$};
	\foreach \x/\xtext in {0/0, 1/1, 2/2, 3/3, 4/4, 5/5}
    	\draw[shift={(\x+0.5,0)}] (0pt,2pt) -- (0pt,-2pt) node[below] {$\xtext$};
    	\foreach \y/\ytext in {0/0, 1/1, 2/2, 3/3, 4/4, 5/5}
    	\draw[shift={(0,\y+0.5)}] (2pt,0pt) -- (-2pt,0pt) node[left] {$\ytext$};

    \end{scope}

    \begin{scope}[
            yshift=-115,every node/.append style={
            yslant=0.5,xslant=-1},yslant=0.5,xslant=-1
            ]
              \fill[white,fill opacity=.9] (0,0) rectangle (6,6);
        \draw[black,very thick] (0,0) rectangle (6,6);
        \draw[step=1cm, black] (0,0) grid (6,6);

       \fill[tuerkis] (0.1,0.1) rectangle (0.9,0.9);
       	\fill[tuerkis] (1.1,0.1) rectangle (1.9,0.9);
        \fill[tuerkis] (0.1,1.1) rectangle (0.9,1.9);

    \end{scope}

     \begin{scope}[
            yshift=-70,every node/.append style={
            yslant=0.5,xslant=-1},yslant=0.5,xslant=-1
            ]
                  \fill[white,fill opacity=.9] (0,0) rectangle (6,6);
        \draw[black,very thick] (0,0) rectangle (6,6);
        \draw[step=1cm, black] (0,0) grid (6,6);
       \fill[tuerkis] (0.1,0.1) rectangle (0.9,0.9);
       	\fill[tuerkis] (1.1,0.1) rectangle (1.9,0.9);
	\fill[tuerkis] (2.1,0.1) rectangle (2.9,0.9);
        \fill[tuerkis] (0.1,1.1) rectangle (0.9,1.9);
        \fill[tuerkis] (1.1,1.1) rectangle (1.9,1.9);
        \fill[tuerkis] (0.1,2.1) rectangle (0.9,2.9);

    \end{scope}

 \begin{scope}[
            yshift=-25,every node/.append style={
            yslant=0.5,xslant=-1},yslant=0.5,xslant=-1
            ]
                  \fill[white,fill opacity=.9] (0,0) rectangle (6,6);
        \draw[black,very thick] (0,0) rectangle (6,6);
        \draw[step=1cm, black] (0,0) grid (6,6);
       \fill[tuerkis] (0.1,0.1) rectangle (0.9,0.9);
       	\fill[tuerkis] (1.1,0.1) rectangle (1.9,0.9);
	\fill[tuerkis] (2.1,0.1) rectangle (2.9,0.9);
	\fill[tuerkis] (3.1,0.1) rectangle (3.9,0.9);
        \fill[tuerkis] (0.1,1.1) rectangle (0.9,1.9);
        \fill[tuerkis] (1.1,1.1) rectangle (1.9,1.9);
        \fill[tuerkis] (2.1,1.1) rectangle (2.9,1.9);
        \fill[tuerkis] (0.1,2.1) rectangle (0.9,2.9);
         \fill[tuerkis] (1.1,2.1) rectangle (1.9,2.9);
         \fill[tuerkis] (0.1,3.1) rectangle (0.9,3.9);

    \end{scope}

    \begin{scope}[
    	yshift=20,every node/.append style={
    	    yslant=0.5,xslant=-1},yslant=0.5,xslant=-1
    	             ]
	    \fill[white,fill opacity=0.9] (0,0) rectangle (6,6);
        \draw[step=1cm, black] (0,0) grid (6,6); %
        \draw[black,very thick] (0,0) rectangle (6,6);%
        \fill[tuerkis] (0.1,0.1) rectangle (0.9,0.9);
       	\fill[tuerkis] (1.1,0.1) rectangle (1.9,0.9);
	\fill[tuerkis] (2.1,0.1) rectangle (2.9,0.9);
	\fill[tuerkis] (3.1,0.1) rectangle (3.9,0.9);
	\fill[tuerkis] (4.1,0.1) rectangle (4.9,0.9);
        \fill[tuerkis] (0.1,1.1) rectangle (0.9,1.9);
        \fill[tuerkis] (1.1,1.1) rectangle (1.9,1.9);
        \fill[tuerkis] (2.1,1.1) rectangle (2.9,1.9);
        \fill[tuerkis] (3.1,1.1) rectangle (3.9,1.9);
        \fill[tuerkis] (0.1,2.1) rectangle (0.9,2.9);
        \fill[tuerkis] (1.1,2.1) rectangle (1.9,2.9);
        \fill[tuerkis] (2.1,2.1) rectangle (2.9,2.9);
        \fill[tuerkis] (0.1,3.1) rectangle (0.9,3.9);
        \fill[tuerkis] (1.1,3.1) rectangle (1.9,3.9);
        \fill[tuerkis] (0.1,4.1) rectangle (0.9,4.9);

    \end{scope}

	 \begin{scope}[
    	yshift=65,every node/.append style={
    	    yslant=0.5,xslant=-1},yslant=0.5,xslant=-1
    	             ]
	    \fill[white,fill opacity=0.9] (0,0) rectangle (6,6);
        \draw[step=1cm, black] (0,0) grid (6,6); %
        \draw[black,very thick] (0,0) rectangle (6,6);%
        \fill[tuerkis] (0.1,0.1) rectangle (0.9,0.9);
       	\fill[tuerkis] (1.1,0.1) rectangle (1.9,0.9);
	\fill[tuerkis] (2.1,0.1) rectangle (2.9,0.9);
	\fill[tuerkis] (3.1,0.1) rectangle (3.9,0.9);
	\fill[tuerkis] (4.1,0.1) rectangle (4.9,0.9);
	\fill[tuerkis] (5.1,0.1) rectangle (5.9,0.9);
        \fill[tuerkis] (0.1,1.1) rectangle (0.9,1.9);
        \fill[tuerkis] (1.1,1.1) rectangle (1.9,1.9);
        \fill[tuerkis] (2.1,1.1) rectangle (2.9,1.9);
        \fill[tuerkis] (3.1,1.1) rectangle (3.9,1.9);
	\fill[tuerkis] (4.1,1.1) rectangle (4.9,1.9);
	\fill[tuerkis] (0.1,2.1) rectangle (0.9,2.9);
        \fill[tuerkis] (1.1,2.1) rectangle (1.9,2.9);
        \fill[tuerkis] (2.1,2.1) rectangle (2.9,2.9);
        \fill[tuerkis] (3.1,2.1) rectangle (3.9,2.9);
	\fill[tuerkis] (0.1,3.1) rectangle (0.9,3.9);
        \fill[tuerkis] (1.1,3.1) rectangle (1.9,3.9);
        \fill[tuerkis] (2.1,3.1) rectangle (2.9,3.9);
        \fill[tuerkis] (0.1,4.1) rectangle (0.9,4.9);
        \fill[tuerkis] (1.1,4.1) rectangle (1.9,4.9);
        \fill[tuerkis] (0.1,5.1) rectangle (0.9,5.9);
    \end{scope}

    	\draw[-latex,thick] (8,7.4) node[right]{$\mathsf{\alpha_3=0}$}
         	to[out=180,in=90] (4.5,4.6);

    	\draw[-latex,thick](8,5.2)node[right]{$\mathsf{\alpha_3=1}$}
        		to[out=180,in=90] (4.5,3);

        \draw[-latex,thick] (8,3) node[right]{$\mathsf{\alpha_3=2}$}
         	to[out=180,in=90] (4.5,1.4);

    	\draw[-latex,thick](8,0.8)node[right]{$\mathsf{\alpha_3=3}$}
        		to[out=180,in=90] (4.5,-0.2);
	\draw[-latex,thick] (8,-1.4) node[right]{$\mathsf{\alpha_3=4}$}
         	to[out=180,in=90] (4.5,-1.8);

    	\draw[-latex,thick](8,-3.6)node[right]{$\mathsf{\alpha_3=5}$}
        		to[out=180,in=90] (4.5,-3.4);
\end{tikzpicture}
\caption{\label{fig:indicessparse} The \(21\) indices contained in the sparse grid \(X(5,2)\) on the left
and the \(56\) indices contained in \(X(5,3)\) on the right.}
\end{center}
\end{figure}
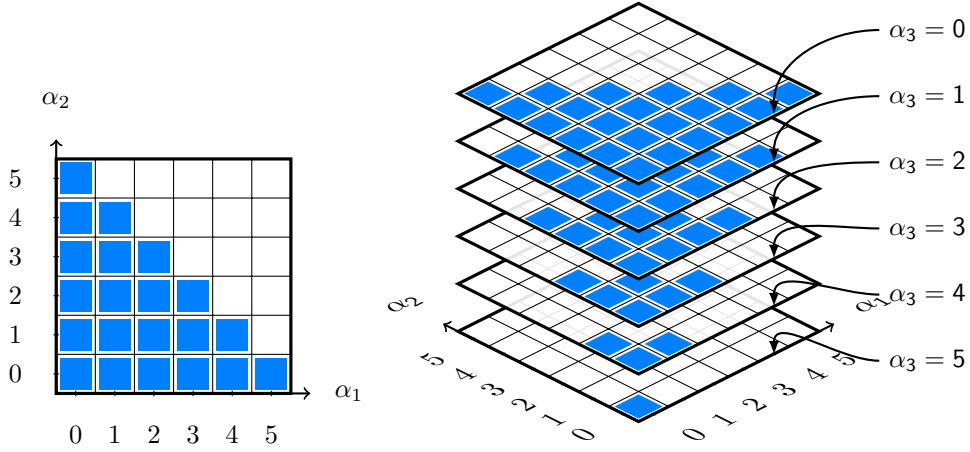

The number of quadrature points used in \eqref{eq:smolop}
or \eqref{eq:smolop2} is considerably reduced compared to 
the full tensor product quadrature. However, the sparse 
quadrature operator does not take into account the fact that the 
different parameter dimensions are of different importance to the 
integrand \(f_m\). Indeed, the cardinality of the set \(X(q,m)\) is 
given by
\[
\# X(q,m)=\binom{q+m}{m},
\]
which still depends heavily on \(m\). 
Thus, we assign a weight to each particular dimension and
use a weighted version of the Smolyak quadrature operator.

Let \({\bf w}\in \mathbb{R}_+^m\) denote a weight vector for the
different parameter dimensions. We assume
in the following that the weight vector is sorted in ascending order,
i.e.~\(w_1\le w_2\le \ldots\le w_m\). Otherwise, we would rearrange
the particular dimensions accordingly.  We modify the sparse grid index
sets \(X(q,m)\) and \(Y(q,m)\) in the following way, see also \cite{NTW08b},
\begin{equation}\label{eq:anisoset1}
X_{\bf w}(q,m)\isdef\Bigg\{{\bf 0}\le\balpha\in\mathbb{N}^m:
\sum_{n=1}^m \alpha_n w_n\le q\Bigg\}
\end{equation}
and
\begin{equation}\label{eq:anisoset2}
Y_{\bf w}(q,m)\isdef\Bigg\{{\bf 0}\le\balpha\in\mathbb{N}^m:
q-\|{\bf w}\|_1<\sum_{n=1}^m \alpha_n w_n\le q
\Bigg\}.
\end{equation}
With this notation at hand, the \emph{anisotropic sparse grid 
quadrature operator} of level \(q\in\mathbb{N}\) is defined by
\begin{equation}\label{eq:anismolop}
\mathcal{A}_{\bf w}(q,m)\isdef\sum_{\balpha\in X_{\bf w}(q,m)}
\Delta_{\alpha_1}^{(1)}\otimes\cdots\otimes\Delta_{\alpha_m}^{(m)}
\end{equation}
which can equivalently be expressed as, cf.~\cite{NTW08b},
\begin{equation}\label{eq:anismolop2}
\mathcal{A}_{\bf w}(q,m)=\sum_{\balpha\in Y_{\bf w}(q,m)}
c_{\bf w}(\balpha)
{\bf Q}_{\balpha},
\quad\text{with }\quad
c_{\bf w}(\balpha)\isdef \sum_{\genfrac{}{}{0pt}{}{\bbeta \in\{0,1\}^m}
{\balpha+\bbeta\in Y_{\bf w}(q,m)}}(-1)^{\|\bbeta\|_1}.
\end{equation}

The formula \eqref{eq:anismolop2} can be regarded as the \emph{anisotropic
combination technique quadrature}. For the evaluation of this formula,
we only need to determine the coefficients \(c_{\bf w}(\balpha)\) and to apply
tensor product quadrature operators of relatively small size.
Thus, in order to compute the approximation to \eqref{eq:IntFinite} with the
anisotropic sparse grid quadrature \eqref{eq:anismolop2}, it is sufficient to 
evaluate the integrand \(f_m\) on the \emph{anisotropic sparse grid}
\[
\mathcal{J}_{\bf w}(q,m)\isdef \bigcup_{\balpha\in Y_{\bf w}(q,m)}
\theta_{\alpha_1}\times\cdots\times\theta_{\alpha_m}.
\]
Note that the sparse grid quadrature operator \eqref{eq:smolop}
coincides with the anisotropic sparse grid quadrature operator
\eqref{eq:anismolop} for the special weight vector \({\bf w}={\bf 1}\isdef[1,1,\ldots,1]\).

In Figure \ref{fig:indicesweighted}, the indices of the weighted sparse
grid \(X_{(1,2.5)}(5,2)\) and of the weighted sparse grid \(X_{(1,2,3)}(5,3)\)
are visualized. We observe that the number of indices is drastically reduced
in comparison to the according isotropic sparse grids visualized in
Figure \ref{fig:indicessparse}.

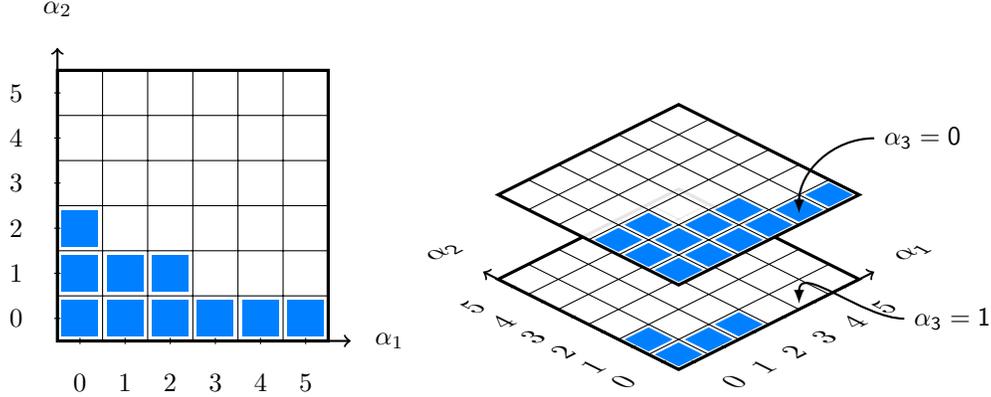
\begin{figure}
\begin{center}
\begin{tikzpicture}[scale=.6,every node/.style={minimum size=1cm},on grid]
 \fill[white,fill opacity=0.9] (0,0) rectangle (6,6);
        \draw[step=1cm, black] (0,0) grid (6,6); %
        \draw[black,very thick] (0,0) rectangle (6,6);%
        \fill[tuerkis] (0.1,0.1) rectangle (0.9,0.9);
       	\fill[tuerkis] (1.1,0.1) rectangle (1.9,0.9);
	\fill[tuerkis] (2.1,0.1) rectangle (2.9,0.9);
	\fill[tuerkis] (3.1,0.1) rectangle (3.9,0.9);
	\fill[tuerkis] (4.1,0.1) rectangle (4.9,0.9);
	\fill[tuerkis] (5.1,0.1) rectangle (5.9,0.9);
        \fill[tuerkis] (0.1,1.1) rectangle (0.9,1.9);
        \fill[tuerkis] (1.1,1.1) rectangle (1.9,1.9);
        \fill[tuerkis] (2.1,1.1) rectangle (2.9,1.9);
        \fill[tuerkis] (0.1,2.1) rectangle (0.9,2.9);

 \draw[->,thick] (0,0) -- (6.5,0) node[right] {$\alpha_1$};
    	\draw[->,thick] (0,0) -- (0,6.5) node[above] {$\alpha_2$};
    	\foreach \x/\xtext in {0/0, 1/1, 2/2, 3/3, 4/4, 5/5}
    	\draw[shift={(\x+0.5,0)}] (0pt,2pt) -- (0pt,-2pt) node[below] {$\xtext$};
    	\foreach \y/\ytext in {0/0, 1/1, 2/2, 3/3, 4/4,5/5}
    	\draw[shift={(0,\y+0.5)}] (2pt,0pt) -- (-2pt,0pt) node[left] {$\ytext$};

\end{tikzpicture}
\hspace*{-1cm}
\begin{tikzpicture}[scale=.4,every node/.style={minimum size=1cm},on grid]
    \begin{scope}[
            yshift=-105,every node/.append style={
            yslant=0.5,xslant=-1},yslant=0.5,xslant=-1
            ]
                  \fill[white,fill opacity=.9] (0,0) rectangle (6,6);
        \draw[black,very thick] (0,0) rectangle (6,6);
        \draw[step=1cm, black] (0,0) grid (6,6);
       \fill[tuerkis] (0.1,0.1) rectangle (0.9,0.9);
       	\fill[tuerkis] (1.1,0.1) rectangle (1.9,0.9);
	\fill[tuerkis] (2.1,0.1) rectangle (2.9,0.9);
        \fill[tuerkis] (0.1,1.1) rectangle (0.9,1.9);

        \draw[->,thick] (0,0) -- (6.5,0) node[right] {$\mathsf{\alpha_1}$};
    	\draw[->,thick] (0,0) -- (0,6.5) node[above] {$\mathsf{\alpha_2}$};
    	\foreach \x/\xtext in {0/0, 1/1, 2/2, 3/3, 4/4, 5/5}
    	\draw[shift={(\x+0.5,0)}] (0pt,2pt) -- (0pt,-2pt) node[below] {$\xtext$};
    	\foreach \y/\ytext in {0/0, 1/1, 2/2, 3/3, 4/4,5/5}
    	\draw[shift={(0,\y+0.5)}] (2pt,0pt) -- (-2pt,0pt) node[left] {$\ytext$};
    \end{scope}

    \begin{scope}[
    	yshift=-25,every node/.append style={
    	    yslant=0.5,xslant=-1},yslant=0.5,xslant=-1
    	             ]
	    \fill[white,fill opacity=0.9] (0,0) rectangle (6,6);
        \draw[step=1cm, black] (0,0) grid (6,6); %
        \draw[black,very thick] (0,0) rectangle (6,6);%
        \fill[tuerkis] (0.1,0.1) rectangle (0.9,0.9);
       	\fill[tuerkis] (1.1,0.1) rectangle (1.9,0.9);
	\fill[tuerkis] (2.1,0.1) rectangle (2.9,0.9);
	\fill[tuerkis] (3.1,0.1) rectangle (3.9,0.9);
	\fill[tuerkis] (4.1,0.1) rectangle (4.9,0.9);
	\fill[tuerkis] (5.1,0.1) rectangle (5.9,0.9);
        \fill[tuerkis] (0.1,1.1) rectangle (0.9,1.9);
        \fill[tuerkis] (1.1,1.1) rectangle (1.9,1.9);
        \fill[tuerkis] (2.1,1.1) rectangle (2.9,1.9);
        \fill[tuerkis] (3.1,1.1) rectangle (3.9,1.9);
        \fill[tuerkis] (0.1,2.1) rectangle (0.9,2.9);
        \fill[tuerkis] (1.1,2.1) rectangle (1.9,2.9);
    \end{scope}

    \draw[-latex,thick] (6.5,4) node[right]{$\mathsf{\alpha_3=0}$}
         to[out=180,in=90] (4,1.5);

    \draw[-latex,thick](7.5,-2)node[right]{$\mathsf{\alpha_3=1}$}
        to[out=180,in=90] (4,-1.5);
\end{tikzpicture}
\caption{\label{fig:indicesweighted} The \(10\) indices contained in the weighted sparse grid
\(X_{(1,2.5)}(5,2)\)
on the left and the \(16\) indices contained in \(X_{(1,2,3)}(5,3)\) on the right.}
\end{center}
\end{figure}

The computation of the anisotropic sparse grid quadrature operator
\eqref{eq:anismolop} depends on the choice of the weight vector
\({\bf w}\) and the sequence \(\big\{N_j\big\}_j\) in \eqref{eq:interpp}.
In view of the one-dimensional error estimate \eqref{eq:assgaussquad} 
and to keep the number of quadrature points low, the sequence 
\(\big\{N_j\big\}_j\) is chosen slowly increasing in accordance with
\begin{equation}\label{eq:choicequadsequence}
N_j=\bigg\lceil\frac{1}{2}(j+2)\bigg\rceil.
\end{equation}
This choice implies \(N_j=N_{j+1}\) for \(j\) odd. 
Therefore, many tensor products of difference quadratures 
in \eqref{eq:anismolop} vanish, see  Remark~\ref{rem:notSharp}.

Then, we can find an upper bound for the contribution of the difference 
quadrature operator
\(\Delta_{j}= \intQ_{j}-\intQ_{j-1}\) for all \(j\ge1\) and for all
functions \(f_1\colon\Gamma\to \mathbb{R}\) which
are analytically extendable into \(\Sigma_1\) according to
\begin{equation}\label{eq:singleerr}
\begin{aligned}
|\Delta_{j}f_1| &\le |\intI f_1-\intQ_{j} f_1|+|\intI f_1-\intQ_{j-1} f_1|\\
&\le c(\kappa_1) \Big(e^{-\log(\kappa_1)(j+1)}+e^{-\log(\kappa_1) j}\Big)
\|f_1\|_{C(\Sigma_1)}\\
&\le c(\kappa_1)\Big(1+e^{-\log(\kappa_1)}\Big)e^{-\log(\kappa_1) j}\|f_1\|_{C(\Sigma_1)}\\
&\le
\frac{\kappa_1+1}{\kappa_1}c(\kappa_1)e^{-\log(\kappa_1) j}\|f_1\|_{C(\Sigma_1)} 
\le 2 c(\kappa_1)e^{-\log(\kappa_1) j}\|f_1\|_{C(\Sigma_1)} .
\end{aligned}
\end{equation}
For \(j=0\), the difference quadrature operator coincides
with the function evaluation at the midpoint \(z=0\) of \(\Gamma\) which implies that
\begin{equation}\label{eq:singleerr0}
|\Delta_{0}f_1|=|\intQ_{0} f_1|=|f_1(0)|\le
e^{-\log(\kappa_1)\cdot 0} \|f_1\|_{C(\Sigma_1)}.
\end{equation}
Analogously, it follows from \eqref{eq:assgaussquad} and \eqref{eq:choicequadsequence}
 that 
\begin{equation}\label{eq:singleerr2}
|\intI f_1-\intQ_{j} f_1|\le c(\kappa_1) e^{-\log(\kappa_1) (j+1)}
\|f_1\|_{C(\Sigma_1)}.
\end{equation}

Next, let us consider the multivariate integrand \(f_m\colon\Gamma^m\to\mathbb{R}\)
which can be analytically extended into the region
\(\boldsymbol{\Sigma}_m\).
Due to the analyticity in a tensor product domain, it follows that the contribution 
of the tensor product of the operators \(\Delta_{j}\) is bounded by the product of the
one-dimensional contributions.
Indeed, we obtain that
\begin{equation}\label{eq:multerr}
\begin{aligned}
&\Big|\Big(\Delta_{\alpha_1}^{(1)}\otimes\cdots\otimes
\Delta_{\alpha_m}^{(m)}\Big)f_m\Big|\\
&\qquad\le\big(2c(\kappa_n)\big)^{\min(1,\alpha_1)}
e^{-\log(\kappa) \alpha_1}\sup_{z\in\Sigma_1}
\Big|\Big(\operatorname{Id}\otimes\Delta_{\alpha_2}^{(2)}\otimes\cdots\otimes
\Delta_{\alpha_m}^{(m)}\Big)f_m(z)\Big|
\\
&\qquad\le \bigg(\prod_{n=1}^m 
\big(2c(\kappa_n)\big)^{\min(1,\alpha_n)}\bigg)
e^{-\sum_{n=1}^m \log(\kappa_n)\alpha_n } \|f_m\|_
{C(\boldsymbol{\Sigma}_m)}
\end{aligned}
\end{equation}
with \( \|f_m\|_{C(\boldsymbol{\Sigma}_m)}
\isdef \sup_{{\bf z}\in \boldsymbol{\Sigma}_m}|f_m({\bf z})|\).
In addition, we take the minimum in \eqref{eq:multerr} in order to ensure that the constant
is \(1\) if \(\alpha_n=0\) in accordance with \eqref{eq:singleerr0}.

\section{Error estimation for the anisotropic sparse grid quadrature}
\label{sec:errana}
For the estimation of the quadrature error of the anisotropic 
sparse grid quadrature, we employ the following lemma.

\begin{lemma}\label{lem:costhsumm}
Let \(\{\psi_n\}_n\in\ell^1(\mathbb{N}^*)\) be a summable sequence of positive real numbers.
Then, there exists for each \(\delta>0\) a constant \(c(\delta)\) independent of \(q\ge 1\)
such that
\begin{equation}\label{eq:boundprod}
\prod_{n=1}^\infty (q \psi_n+1)
\le c(\delta) \exp(\delta q).
\end{equation}
\end{lemma}

\begin{proof}
Let \(0<\delta_1,\delta_2<\delta\) be arbitrary such that \(\delta_1
+\delta_2=\delta\). From the summability of \(\{\psi_n\}_n\), it follows 
that there exists a \(j_0 = j_0(\delta_1)\in \mathbb{N}\) such that 
\begin{equation}\label{eq:truncwithj}
\sum_{n=j_0+1}^{\infty}\psi_n \le \delta_1.
\end{equation}
We now split the left-hand side in \eqref{eq:boundprod} into
\begin{equation}\label{eq:costcompGH}
\prod_{n=1}^\infty ( q\psi_n+1)=
\prod_{n=1}^{j_0} ( q\psi_n+1)
\prod_{{n}=j_0+1}^\infty ( q\psi_n+1).
\end{equation}
Then, the second factor can simply be estimated by
\begin{align*}
\prod_{n=j_0+1}^\infty ( q\psi_n+1)= \exp\bigg(\sum_{n=j_0+1}^\infty \log(q\psi_n+1)\bigg)\le
\exp(\delta_1 q).
\end{align*}
The number of factors \(j_0\) in the first product on the right-hand side
of \eqref{eq:costcompGH} 
is fixed and depends only on the choice of \(\delta_1\) and on the decay 
properties of \(\{\psi_k\}_k\). Since \(j_0\) is a fixed natural number, there 
exists for all \(\delta_2>0\) a constant \(c(\delta_1,\delta_2)\) such that
\[
\prod_{{n}=1}^{j_0} (q\psi_n+1)
\le c(\delta_1,\delta_2) \exp(\delta_2 q).
\]
Hence, we obtain that
\[
\prod_{n=1}^\infty ( q\psi_n+1)\le c(\delta_1,\delta_2) \exp(\delta q).
\]
Since \(0<\delta_1,\delta_2<\delta\) can be chosen arbitrary with the only limitation that
\(\delta_1+\delta_2=\delta\), the choice \(c(\delta)=\inf_{\delta_1+\delta_2=\delta } c(\delta_1,\delta_2)\)
yields the desired estimate.
\end{proof}

With the above preliminaries and further assumptions on the summability 
of the sequence \(\{\tau_n\}_n\), we are able to establish error
estimates for the anisotropic sparse grid quadrature.  

\begin{assumption}\label{ass:tauandgandh}
The sequence \(\{\tau_n\}_n\), which describes the regions
of analytic extendability of the function \(f\), fulfills
\[
\tau_n\ge c n^{r}
\]
for some \(r>1\) and a constant \(c>0\). Hence, the sequences \(\{\tau^{-1}_n\}_n\) 
and \(\{(\kappa_n-1)^{-1}\}_n\) are summable.
\end{assumption}

For the error estimation, we apply an identity which was used in 
\cite{NTW08b} to bound the error of a collocation approach. 
Additionally, we exploit Assumption \ref{ass:tauandgandh} to 
obtain an estimate which is exponentially decreasing in the sparse 
grid level \(q\) and does not depend on the dimensionality \(m\) at all. 
Therefore, we further notice that the integration operator 
\({\bf I}\colon C(\boldsymbol{\Sigma}_m)\to \mathbb{R}\) is obviously continuous 
with continuity constant \(1\).

\begin{lemma}\label{lem:recest}
Let the sequence of quadrature points be chosen as in
\eqref{eq:choicequadsequence} and let the weight vector
\({\bf w}\) be given by \(w_n=\log(\kappa_n)\). Then, there 
exists for each \(\delta>0\) a constant \(c(\delta)\) independent 
of \(m\) such that the error of the anisotropic sparse grid quadrature
\eqref{eq:smolop} is bounded by
\begin{equation}\label{eq:errrep}
\big|\big({\bf I}-\mathcal{A}_{\bf w}(q,m)\big) f_m\big| 
\le c(\delta,\btau) e^{-q(1-\delta)} \|f_m\|_{C(\boldsymbol{\Sigma}_m)}
\end{equation} 
with 
\[
c(\delta,\btau) = 4c(\delta)\|\{\tau_n^{-1}\}_n\|_{\ell^1}
\]
and \(c(\delta)\) denotes the constant from Lemma \ref{lem:costhsumm} 
with respect to the summable sequence \(\{\tau_n^{-1}\}_n\). 
Note that the constant \(c(\delta,\btau)\) 
tends to infinity as \(\delta\) tends to \(0\).
\end{lemma}
\begin{proof}
In the same way as in \cite{NTW08b},
the error of the sparse grid quadrature is rewritten,
with the notation \({\bf I}=\bigotimes_{n=1}^m \intI^{(n)}\), by
\begin{equation}\label{eq:reorderingerror}
{\bf I}-\mathcal{A}_{\bf w}(q,m)=\sum_{n=1}^m R(q,n) \bigotimes_{
k=n+1}^{m} \intI^{(k)}.
\end{equation}
Herein, the quantity \(R(q,n)\) is defined for \(n=1\) by
\[
R(q,1)\isdef \intI^{(1)}-\intQ_{\lfloor q/w_1 \rfloor}
\]
and for \(n\ge 2\) by
\[
R(q,n)\isdef\sum_{\balpha\in X_{{\bf w}_{1:n-1}}(q,n-1)}
\bigotimes_{k=1}^{n-1} \Delta_{\alpha_k}^{(k)}
\otimes\bigg(\intI^{(n)}-\intQ_{\big\lfloor \big(q-
\sum_{k=1}^{n-1}\alpha_k w_{k}\big)/w_n\big\rfloor}\bigg).
\]

Due to \(R(q,1)=\intI^{(1)}-\intQ_{\lfloor {q}/{w_1}\rfloor}\), 
we deduce for the first summand in \eqref{eq:reorderingerror} 
that
\[
\bigg|\bigg(R(q,1)\bigotimes_{k=2}^{m} \intI^{(k)}\bigg)
f_m\bigg|
\le c(\kappa_1) e^{-\log(\kappa_1)(\lfloor
{q}/{w_1}\rfloor+1)}\|f_m\|_{C(\boldsymbol{\Sigma}_m)}
\le c(\kappa_1) e^{-q}\|f_m\|_{C(\boldsymbol{\Sigma}_m)}.
\]
The summands in \eqref{eq:reorderingerror} with \(n\ge2\) can be 
estimated with \eqref{eq:singleerr2}, \eqref{eq:multerr} and with the 
continuity of the integration operator by 
\begin{align*}
\Bigg|\Bigg( R(q,n) \bigotimes_{k=n+1}^{m} \intI^{(k)}\Bigg)f_m\Bigg|
&\le
\sum_{\balpha\in X_{{\bf w}_{1:n-1}}(q,n-1)}
\bigg(\prod_{k=1}^{n-1} \big(2c(\kappa_k)\big)^{\min(1,\alpha_k)}\bigg) e^{-\sum_{k=1}^{n-1} \alpha_k \log(\kappa_k)}\\
&\phantom{\le}\qquad\qquad\qquad\cdot c(\kappa_n) e^{-\log(\kappa_n) \big(\big\lfloor
\big(q-\sum_{k=1}^{n-1}\alpha_{k} w_{k}\big)/w_n\big\rfloor+1\big)}
\|f_m\|_{C(\boldsymbol{\Sigma}_m)}\\
&\le  c(\kappa_n) \sum_{\balpha\in X_{{\bf w}_{1:n-1}}(q,n-1)}
e^{-\log(\kappa_n) \big(\big\lfloor
\big(q-\sum_{k=1}^{n-1}\alpha_{k} w_{k}\big)/w_n\big\rfloor+1\big)-
\sum_{k=1}^{n-1} \alpha_k  \log(\kappa_k)}\\
&\phantom{\le}\qquad\qquad\qquad
\cdot\bigg(\prod_{k=1}^{n-1} \big(2c(\kappa_k)\big)^{\min(1,\alpha_k)}\bigg)\|f_m\|_{C(\boldsymbol{\Sigma}_m)}.
\end{align*}
With the choice \(w_k=\log(\kappa_k)\) for all \(k=1,\ldots,m\), it follows that
\begin{align*}
&\bigg|\bigg( R(q,n) \bigotimes_{k=n+1}^{m}\intI^{(k)}\bigg)f_m\bigg|\\
&\qquad \le c(\kappa_n)\sum_{\balpha\in X_{{\bf w}_{1:n-1}}(q,n-1)}
e^{-q-\sum_{k=1}^{n-1}\alpha_{k}w_{k}+\sum_{k=1}^{n-1}
\alpha_{k} w_{k}}
\bigg(\prod_{k=1}^{n-1}\big(2c(\kappa_k)\big)^{\min(1,\alpha_k)}\bigg)\|f_m\|_{C(\boldsymbol{\Sigma}_m)}\\
&\qquad=c(\kappa_n)\sum_{\balpha\in X_{{\bf w}_{1:n-1}}(q,n-1)}e^{-q}
\bigg(\prod_{k=1}^{n-1} \big(2c(\kappa_k)\big)^{\min(1,\alpha_k)}\bigg)\|f_m\|_{C(\boldsymbol{\Sigma}_m)}.
\end{align*}
It remains to estimate
\[
\sum_{\balpha\in X_{{\bf w}_{1:n-1}}(q,n-1)}\bigg(\prod_{k=1}^{n-1} 
\big(2c(\kappa_k)\big)^{\min(1,\alpha_k)}\bigg)
\le \sum_{\balpha\in X_{{\bf w}}(q,m)}\bigg(\prod_{k=1}^{m} 
\big(2c(\kappa_k)\big)^{\min(1,\alpha_k)}\bigg).
\]
The expression inside the product always equals \(2c(\kappa_k)\) except for the case \(\alpha_k=0\).
Hence, it follows that
\begin{align*}
\sum_{\balpha\in X_{{\bf w}}(q,m)}\bigg(\prod_{k=1}^{m} 
\big(2c(\kappa_k)\big)^{\min(1,\alpha_k)}\bigg)
&\le \sum_{\alpha_1=0}^{\lfloor \frac{q}{w_1}\rfloor} \big(2c(\kappa_1)\big)^{\min(\alpha_1,1)}
\cdots
\sum_{\alpha_{m}=0}^{\lfloor \frac{q}{w_{m}}\rfloor}\big(2c(\kappa_m)\big)^{\min(\alpha_{m},1)}\\
&\le \prod_{k=1}^{m} 
\bigg(\frac{2c(\kappa_k) q}{w_k}+1\bigg)\le  
c(\delta) \exp(\delta q).
\end{align*}
The last inequality holds since \(\{2c(\kappa_k)/w_k\}_k\) is summable and, thus, Lemma
\ref{lem:costhsumm} is applicable. 

Combining our findings, we obtain that 
\begin{align*}
\big|\big({\bf I}-\mathcal{A}_{\bf w}(q,m)\big) f_m\big| 
&\le c(\delta)e^{-q(1-\delta)} \|f_m\|_{C(\boldsymbol{\Sigma}_m)}
\sum_{n=1}^m c(\kappa_n)\\
&\le c(\delta)e^{-q(1-\delta)} \|f_m\|_{C(\boldsymbol{\Sigma}_m)}
\sum_{n=1}^m \frac{4}{\tau_n},
\end{align*}
which yields the estimate \eqref{eq:errrep}. 
\end{proof}

Lemma \ref{lem:recest} implies that the anisotropic sparse grid quadrature
converges exponentially with respect to the level \(q\). The convergence
in Lemma \ref{lem:recest} is nearly as good as the convergence of the
anisotropic tensor product quadrature on level \(q\), with \(\lceil \frac{q}{2w_n} +
\frac{1}{2} \rceil\)
quadrature points in the \(n\)-th dimension.

In addition, we provide an estimate on the quadrature error in terms of the number of 
multiindices contained in the anisotropic sparse index set. Note that this is very similar 
to the analysis in \cite{BNTT12, GO16}. From \eqref{eq:exprint}, we deduce 
that the error of the anisotropic sparse grid quadrature can be written as 
\begin{equation}\label{eq:erraniso}
\big|\big({\bf I}-\mathcal{A}_{\bf w}(q,m)\big) f_m\big| = 
\bigg| \sum_{\balpha\in\mathbb{N}^m\setminus X_{\bf w}(q,m)} 
\big(\Delta_{\alpha_1}^{(1)}\otimes\cdots\otimes
\Delta_{\alpha_m}^{(m)}\big) f_m \bigg|.
\end{equation}
With estimate \eqref{eq:multerr}, we obtain 
\begin{align*}
\big|\big({\bf I}-\mathcal{A}_{\bf w}(q,m)\big) f_m\big|&\le 
\sum_{\balpha\in\mathbb{N}^m\setminus X_{\bf w}(q,m)} 
\bigg(\prod_{n=1}^m \big(2c(\kappa_n)\big)^{\min(1,\alpha_n)}\bigg)
e^{-\sum_{n=1}^m \log(\kappa_n)\alpha_n } \|f_m\|_
{C(\boldsymbol{\Sigma}_m)}\\ 
&\le c({\bkappa})  \|f_m\|_
{C(\boldsymbol{\Sigma}_m)}\sum_{\balpha\in\mathbb{N}^m\setminus X_{\bf w}(q,m)} 
e^{-\sum_{n=1}^m \log(\kappa_n)\alpha_n }
\end{align*}
with 
\[
c({\bkappa})\isdef \sup_{\alpha\in \mathbb{N}^m} 
\prod_{n=1}^m \big(2c(\kappa_n)\big)^{\min(1,\alpha_n)}. 
\]
Due to the summability properties of \(\{\tau_n\}_{n}\), the constant \(c({\bkappa})\) can obviously 
be bounded independent of the dimension \(m\). 
It remains to estimate the sum in the above estimate which has been extensively studied in \cite{GO16}. 
The following result from \cite{GO16} is particularly useful for the considered situation. 
\begin{theorem}\label{thm:GO}
Let the sequence \({\bf w}=\{w_n\}_n\) of positive, real numbers be ordered, i.e.~\(w_i\le w_j\) for \(i\le j\). 
If there exists 
a real number \(\beta>1\) such that 
\begin{equation}\label{eq:constGO}
M({\bf w},\beta)\isdef \sum_{n=1}^\infty \frac{1}{e^{w_n/\beta}-1} <\infty,  
\end{equation}
it holds that 
\begin{equation}\label{eq:algconv}
\sum_{\balpha\in\mathbb{N}_0^{\infty}\setminus X_{\bf w}(q,\infty)} 
e^{-\sum_{n=1}^\infty w_n\alpha_n }\le \frac {1}{\beta} e^{\beta M({\bf w},\beta)} \#  X_{\bf w}(q,\infty)^{-(\beta-1)}. 
\end{equation}
\end{theorem}
We apply Theorem \ref{thm:GO} with respect to the sequence \(w_n = \log(\kappa_n)\). Then, it follows 
from Assumption \ref{ass:tauandgandh} that the conditions of Theorem \(\ref{thm:GO}\) are fulfilled with 
\(\beta < r\) since then 
\begin{equation}\label{eq:summability}
 \sum_{n=1}^\infty \frac{1}{e^{\log(\kappa_n)/\beta}-1} \le 
 \sum_{n=1}^\infty \frac{1}{(1+cn^r)^{1/\beta}-1}<\infty. 
\end{equation}

Of course, Theorem \ref{thm:GO} is still valid when considering dimensions \(m<\infty\). 
In this case, even subexponential convergence rates can be proven which, however, 
depend on the dimension \(m\), see \cite{GO16} for the details. Especially, the appearing constants 
depend then on the dimensionality as well. Since we are interested in dimensionalities 
which might grow with the desired accuracy, it is reasonable to rely on estimates which 
do not depend on \(m\). To that end, we conclude from Theorem \ref{thm:GO} for 
all \(\beta<r\) that 
\begin{equation}\label{eq:errestGO}
\begin{aligned}
\big|\big({\bf I}-\mathcal{A}_{\bf w}(q,m)\big) f_m\big|
 &\le c({\bkappa})
\frac{e^{\beta M^{(m)}({\bf w},\beta)}}{\beta}\#  X_{\bf w}(q,m)^{-(\beta-1)}\|f_m\|_
{C(\boldsymbol{\Sigma}_m)}\\
&\le c({\bf w}, \beta) \#  X_{\bf w}(q,m)^{-(\beta-1)} \|f_m\|_
{C(\boldsymbol{\Sigma}_m)}
\end{aligned}
\end{equation}
with \(M^{(m)}({\bf w},\beta)\isdef \sum_{n=1}^m \frac{1}{e^{w_n/\beta}-1}\). 
Due to \eqref{eq:summability}, the latter constant 
is bounded independently of \(m\).

\section{Cost of the anisotropic sparse grid quadrature}\label{sec:cost}
\subsection{A preliminary estimate on the cost}
In order to find an error estimate in terms of the number of quadrature points, 
we additionally have to estimate the cost of the sparse grid quadrature
method for a given level \(q\).

In the following, we establish a bound on the number of quadrature
points used in the combination technique formula \eqref{eq:anismolop2}. 
This number is obviously bounded by 
\[
\operatorname{cost}\big(\mathcal{A}_{\bf w}(q,m) \big)
\le \sum_{\balpha\in Y_{\bf w}(q,m)} \prod_{n=1}^mN_{\alpha_n},
\]
which might be a rough upper bound since some of the coefficients in 
\eqref{eq:anismolop2} may vanish and since some of the quadrature points 
usually appear repeatedly in \eqref{eq:anismolop2}.  
Since \(Y_{\bf w}(q,m)
\subset X_{\bf w}(q,m)\), cf.~\eqref{eq:anisoset1} and \eqref{eq:anisoset2}, 
we can further estimate 
\begin{equation}\label{eq:numofpoi}
\begin{aligned}
\operatorname{cost}\big(\mathcal{A}_{\bf w}(q,m) \big)
&\le \sum_{\balpha\in Y_{\bf w}(q,m)} \prod_{n=1}^m \bigg\lceil\frac 1 2
(\alpha_n+2) \bigg\rceil
\le \sum_{\balpha\in Y_{\bf w}(q,m)} \prod_{n=1}^m (\alpha_n+1)\\
&\le \sum_{\balpha\in X_{\bf w}(q,m)} \prod_{n=1}^m (\alpha_n+1)
\le \bigg[\max_{\balpha\in X_{\bf w}(q,m)} \prod_{n=1}^m (\alpha_n+
1)\bigg] \#  X_{\bf w}(q,m).
\end{aligned}
\end{equation}
Note that there holds \(\bbeta\in X_{\bf w}(q,m)\) for all \({\bf 0}\le \bbeta \le \balpha\) whenever
\(\balpha\in X_{\bf w}(q,m)\). This property is often referred to as \emph{downward closedness} of the index set.
Hence, it follows that 
\[
 \#  X_{\bf w}(q,m) \ge \max_{\balpha\in X_{\bf w}(q,m)} \sum_{{\bf 0}\le \bbeta\le \balpha} 1 
 = \max_{\balpha\in X_{\bf w}(q,m)} \prod_{n=1}^m (\alpha_n+1). 
\]
As a consequence, the number of quadrature points in \eqref{eq:anisoset2} with a sequence of
univariate quadrature sequence, where the number of points are given by \eqref{eq:choicequadsequence}, 
is bounded by 
\begin{equation}\label{eq:numofpoi2}
\operatorname{cost}\big(\mathcal{A}_{\bf w}(q,m) \big)
\le  \#  X_{\bf w}(q,m)^2.
\end{equation}
We remark that a similar bound for downward closed index sets has also been 
used in \cite{EST16} in the context of a sparse adaptive collocation approximation.

\begin{remark}\label{rem:notSharp}
As indicated before, the estimate \eqref{eq:numofpoi2} is not sharp. Our numerical experiments indicate that
\(\operatorname{cost}\big(\mathcal{A}_{\bf w}(q,m)\big)\) depends rather linearly 
than quadratically on the number of multi-indices \( \#  X_{\bf w}(q,m)\), 
cf.~Figures \ref{fig:QuadProb2}--\ref{fig:QuadProb4} from the numerical examples. 
This is due to the fact that an exact representation for the cost
is given by 
\begin{equation}\label{eq:excost}
\operatorname{cost}\big(\mathcal{A}_{\bf w}(q,m) \big)
= \sum_{\balpha\in X_{\bf w}(q,m)} \prod_{n=1}^m \zeta_{\alpha_n} 
\end{equation}
where \(\zeta_{\alpha_n}\) denotes the number of quadrature points which belong to 
\(\intQ_{\alpha_n}^{(n)}\) but not to \(\intQ_{i}^{(n)}\) for any \(i<\alpha_n\). 
In our setting of the Gauss-Legendre quadrature, where the number of points is determined
by \eqref{eq:choicequadsequence}, this sequence is given by 
\[
\{\zeta_n\}_n = \{1,2,0,2,0,4,0,4,0,6,0,6,\ldots \}.
\]
Therefore, each summand in \eqref{eq:excost} vanishes whenever any \(\alpha_n\) is an even 
number greater than zero. 
\end{remark}

\subsection{An improved estimate on the anisotropic sparse index set}\label{sec:sharp}

In order to complete the convergence analysis, it remains to estimate the number of indices
in the set \(X_{\bf w}(q,m)\). To that end, we require the following lemma.

\begin{lemma}\label{lem:helplem}
For \(L\in \mathbb{N}\), \(m\in\mathbb{N}\) and \(\delta\in \mathbb{R}_+\), there holds the inequality
\[
\sum_{j=0}^{L-1} \prod_{n=1}^{m} (n+\delta+j)\le \frac{1}{m+1}\prod_{n=0}^{m}(L+\delta+n)
\]
with equality when \(\delta=0\).
\end{lemma}

\begin{proof}
We prove the assertion by induction on \(L\). For \(L=1\), we verify
\[
\prod_{n=1}^m(n+\delta) = \frac{1}{m+1+\delta}\prod_{n=1}^{m+1}(n+\delta)
\le \frac{1}{m+1}\prod_{n=0}^{m} (n+\delta+1).
\]
Let the assertion be fulfilled for \(L\). Then, we conclude for \(L+1\) that
\begin{align*}
\sum_{j=0}^{L}\prod_{n=1}^m(n+\delta+j)&\le \frac{L+\delta}{m+1}
\prod_{n=1}^{m}(L+\delta+n)+\prod_{n=1}^m (L+\delta + n)\\
&= \bigg(\frac{L+m+1+\delta}{m+1}\bigg)\prod_{n=1}^m (L+n+\delta)\\
&= \bigg(\frac{1}{m+1}\bigg)\prod_{n=1}^{m+1} (L+n+\delta)\\
&=\bigg(\frac{1}{m+1}\bigg)\prod_{n=0}^{m} (L+1+n+\delta).
\end{align*}
\end{proof}

The next lemma gives us a novel bound on the number of indices in \(X_{\bf w}(q,m)\).
\begin{lemma}\label{conjecture}
The cardinality of the set \(X_{\bf w}(q,m)\) in \eqref{eq:anisoset1},
where the weight vector \({\bf w}=[w_1,\ldots,w_m]\) is ascendingly ordered,
i.e.~\(w_1\le w_2\le\cdots\le w_m\), is bounded by
\begin{equation}\label{eq:estcardset1}
\# {X_{\bf w}(q,m)}\le \prod_{n=1}^m \bigg(\frac{ q}{nw_n}+1\bigg).
\end{equation}
\end{lemma}

\begin{proof}
The prove is performed by induction on \(m\). For \(m=1\), the assertion is obviously
fulfilled, since
\[
\# {X_{\bf w}(q,1)} =\sum_{\alpha_1=0}^{\lfloor
\frac{q}{w_1}\rfloor} 1=\Big\lfloor
\frac{q}{w_1}\Big\rfloor+1.
\]
Let us assume that \eqref{eq:estcardset1} is true for \(m-1\).
For \(m\in\mathbb{N}\), the cardinality of \(X_{\bf w}(q,m)\) can be calculated by
\[
\# {X_{\bf w}(q,m)} = \sum_{j=0}^{\lfloor\frac{q}{w_m}\rfloor}
\# {X_{{\bf w}_{1:m-1}}(q-jw_m,m-1)}.
\]
Inserting the induction hypothesis yields that
\begin{equation}\label{eq:mengerek}
\begin{aligned}
\# {X_{\bf w}(q,m)}&\le \sum_{j=0}^{\lfloor\frac{q}{w_m}\rfloor}
\prod_{n=1}^{m-1}\bigg(1+\frac{q-j w_m}{nw_n}\bigg)\\
&=\Bigg(\prod_{k=1}^{m-1} \bigg(1+\frac{q}{kw_k}\bigg)\Bigg)
\sum_{j=0}^{\lfloor\frac{q}{w_m}\rfloor}\prod_{n=1}^{m-1}
\frac{1+\frac{q-j w_m}{nw_n}}{1+\frac{q}{nw_n}}\\
&= \Bigg(\prod_{n=1}^{m-1} \bigg(1+\frac{q}{nw_n}\bigg)\Bigg)
\sum_{j=0}^{\lfloor\frac{q}{w_m}\rfloor}\prod_{n=1}^{m-1}
\bigg(1-\frac{jw_m}{nw_n+q}\bigg)\\
&= \Bigg(\prod_{n=1}^{m-1} \bigg(1+\frac{q}{nw_n}\bigg)\Bigg)
\Bigg(1+\sum_{j=1}^{\lfloor\frac{q}{w_m}\rfloor}\prod_{n=1}^{m-1}
\bigg(1-\frac{jw_m}{nw_n+q}\bigg)\bigg).
\end{aligned}
\end{equation}
Focusing on the last term and since \(w_m\ge w_n\) for all \(0\le n\le m\),
we conclude that
\begin{align*}
\sum_{j=1}^{\lfloor\frac{q}{w_m}\rfloor}\prod_{n=1}^{m-1}
\bigg(1-\frac{jw_m}{nw_n+q}\bigg)
&\le \sum_{j=1}^{\lfloor\frac{q}{w_m}\rfloor}\prod_{n=1}^{m-1}
\bigg(1-\frac{jw_m}{nw_m+q}\bigg)\\
& =  \prod_{n=1}^{m-1} \bigg(n+\frac{q}{w_m}\bigg)^{-1}
\sum_{j=1}^{\lfloor\frac{q}{w_m}\rfloor}\prod_{n=1}^{m-1}
\bigg(n+\frac{q}{w_m}-j\bigg).
\end{align*}
Applying the previous lemma with \(L= \lfloor \frac{q}{w_m}\rfloor\)
and \(\delta = \frac{q}{w_m}-L\) leads to
\[
\sum_{j=1}^{L}\prod_{n=1}^{m-1} (n+L+\delta-j)=
\sum_{j=0}^{L-1}\prod_{n=1}^{m-1} (n+\delta+j)\le
\frac{1}{m}\prod_{n=0}^{m-1}(L+\delta+n).
\]
Thus, we obtain that
\[
\sum_{j=1}^{\lfloor\frac{q}{w_m}\rfloor}\prod_{n=1}^{m-1}
\bigg(1-\frac{jw_m}{nw_n+q}\bigg)\bigg)
\le \frac{L+\delta}{m}=\frac{q}{mw_m}.
\]
Inserting this into \eqref{eq:mengerek} finishes the proof.
\end{proof}

\begin{remark}\label{rem:indexset}
\begin{enumerate}

\item
We would like to point out that estimate \eqref{eq:estcardset1}
is sharp in the isotropic case, that is, for the weight \({\bf w}={\bf 1}\).
Moreover, the ordering of the weight vector is crucial in this estimate.
There are examples where this estimate does not hold if the weights
are not in ascending order.

\item
At first glance one might claim that even
the estimate
\[
\# {X_{\bf w}(q,m)}\le \prod_{n=1}^m \frac{\big\lfloor
\frac{q}{w_n}\big\rfloor+n}{n}
\]
is valid. This is true in a lot of cases which we investigated. Nonetheless,
there are examples where this estimate fails.
\end{enumerate} 
\end{remark}

The novel estimate 
\begin{equation}\label{eq:novest}
\# {X_{\bf w}(q,m)}\le \prod_{n=1}^m \frac{\frac{ q}{w_n}+n}{n}\tag{SG Formula}
\end{equation}
is much more accurate than the well 
established and widely used formula by Beged-Dov, cf.~\cite{Beg72},
\begin{equation}\label{eq:Begdov}
\# {X_{\bf w}(q,m)}\le \prod_{n=1}^m \frac{q+\|{\bf w}\|_1}{nw_n} \tag{BD Formula}
\end{equation}
or the anisotropic tensor product estimate 
\begin{equation}\label{eq:anisotpest}
\# {X_{\bf w}(q,m)}\leq\prod_{n=1}^m\bigg(\bigg\lfloor\frac{q}{w_n}\bigg\rfloor+1\bigg)\tag{TP Formula},
\end{equation}
Indeed, the new estimate implies \eqref{eq:Begdov}. This can easily be shown by induction: For \(m=1\), both estimates coincide. The induction step \(m\mapsto m+1\) follows with \(\tilde{q} = q+\sum_{n=1}^m w_n\) according
to
\begin{align*}
\prod_{n=1}^{m+1} (q+\|{\bf w}\|_1) 
&= (\tilde{q}+w_{m+1})^{m+1} = \sum_{n=0}^{m+1} \binom{m+1}{n} \tilde{q}^nw_{m+1}^{m+1-n} \\
& \ge \tilde{q}^{m+1} + (m+1) \tilde{q}^m w_{m+1} \\
&\ge \big(q+(m+1) w_{m+1}\big) \tilde{q}^m\\
&\ge \big(q+(m+1) w_{m+1}\big) \prod_{n=1}^m (q+nw_n)= \prod_{n=1}^{m+1} (q+nw_n).
\end{align*}
A numerical comparison of the exact number of indices, of our new bound and 
of the formula of Beged-Dov can be found in Figure \ref{fig:PointEstimates} in the numerical examples. 

In practical applications, we will usually have to choose the level 
\(q\). Hence, it is also interesting to examine how the computational cost behave, under the 
decay Assumption \ref{ass:tauandgandh}, with respect to the dimension \(m\). 

\begin{lemma}\label{lem:estpointslog}
Let Assumption \ref{ass:tauandgandh} hold for some \(c>1\) and let \(m\ge 3\). 
Then, we obtain that the number of indices in the anisotropic sparse grid 
on level \(q\) is bounded by
\begin{equation}\label{eq:estalgdecindic}
\# {X_{\bf w}(q,m)}\le c(r) \exp\bigg(\frac{q}{r} \log(\log(m))\bigg) 
 = c(r) \log(m)^{q/r}
\end{equation}
with a constant \(c(r)\) which is independent of \(m\).
\end{lemma}

\begin{proof}
From Lemma \ref{conjecture}, we know that
\[
\# {X_{\bf w}(q,m)}\le \prod_{n=1}^m \bigg(\frac{q}{n w_n}+1\bigg).
\]
Next, we split the product into
\begin{equation}\label{eq:prodsplit}
\prod_{n=1}^m  \bigg(\frac{q}{n w_n}+1\bigg)=\bigg(\frac{q}{w_1}+1\bigg)\bigg(\frac{q}{2 w_2}+1\bigg)
\bigg(\frac{q}{3 w_3}+1\bigg)
\prod_{n=4}^m \bigg(\frac{q}{n w_n}+1\bigg).
\end{equation}
We estimate the last term by
\[
\prod_{n=4}^m \bigg(\frac{q}{n w_n}+1\bigg)=
\exp\bigg(\sum_{n=4}^m \log\bigg(\frac{q}{n w_n}+1\bigg)\bigg)
\le \exp\bigg(\sum_{n=4}^m \frac{q}{n w_n}\bigg).
\]
Due to $w_n\ge\log(n^{r})$, the sum in this
estimate can be bounded by the following integral:
\begin{align*}
\sum_{n=4}^m \frac{q}{n w_n}\le \int_{3}^m \frac{q}{x \log(x^r)} \d x &=\frac{q}{r} \int_{3}^{m}
\frac{1}{x \log(x)} \d x\\
 &=\frac{q}{r} \int_{\log(3)}^{\log(m)} \frac{1}{z} \d z
=\frac{q}{r} \big(\log(\log(m))-\log(\log(3))\big).
\end{align*}
The first three factors in \eqref{eq:prodsplit} define a
cubic polynomial in $q$ and can thus be estimated by the
exponential function according to
\[
\bigg(\frac{q}{w_1}+1\bigg)\bigg(\frac{q}{2 w_2}+1\bigg)
\bigg(\frac{q}{3 w_3}+1\bigg)
\le c(r) \exp\bigg(\frac{\log(\log(3))}{r} q\bigg).
\]
Hence, putting all together, we end up with
\begin{align*}
\prod_{n=1}^m  \bigg(\frac{q}{n w_n}+1\bigg)&\le c(r) \exp\bigg(\frac{\log(\log(3))}{r} q\bigg)
\exp\bigg(\frac{q}{r} \big(\log(\log(m))-\log(\log(3))\big)\bigg)\\
& \le c(r)
\exp\bigg(\frac{q}{r} \big(\log(\log(m))\bigg).
\end{align*}
\end{proof}

\begin{remark}
It follows immediately from Lemma \ref{lem:costhsumm} and the novel 
upper bound \eqref{eq:novest} that $X_{\bf w}(q,m)$ is bounded 
independently of $m$ whenever $\{(kw_k)^{-1}\}_k$ is summable. 
This result cannot be proven by \eqref{eq:Begdov}. Indeed, for 
$w_k=k^r$ with $r>0$, it holds that 
\begin{align*}
\frac{(q+\sum_{k=1}^m w_k)^m}{m! \prod_{k=1}^m w_k}&>1.1 \frac{(q+\sum_{k=1}^m k^r)^m}{(2\pi m)^{(r+1)/2}(m/e)^{m(r+1)}}
> 1.1 \frac{(q+m^{r+1}/(r+1))^m}{(2\pi m)^{(r+1)/2}(m/e)^{m(r+1)}}\\
&>1.1 \frac{m^{m(r+1)}e^{m(r+1)}}{(2\pi m)^{(r+1)/2}m^{m(r+1)}(r+1)^m}=1.1 \frac{1}{(2\pi m)^{(r+1)/2}}\Big(\frac{e^{r+1}}{r+1}\Big)^m. 
\end{align*}
Since \(f(x) =e^x/x\) is strictly increasing for $x>1$, we obtain
\[
\frac{e^{r+1}}{r+1}>e. 
\]
Hence, it follows that
\[
\frac{1}{(2\pi m)^{(r+1)/2}}\Big(\frac{e^{r+1}}{r+1}\Big)^m>\frac{\exp(m)}{(2\pi m)^{(r+1)/2}}
\]
which implies that \eqref{eq:Begdov} grows at least exponentially 
in $m$ for weights $w_k=k^r$ with $r>0$. 
\end{remark}

\subsection{Convergence in terms of the number of quadrature points}
The findings from the previous two sections can be summarized to an error 
estimate in terms of the sparse grid quadrature level \(q\), that is
\begin{equation}\label{eq:errestlevelq}
\big|\big({\bf I}-\mathcal{A}_{\bf w}(q,m)\big) f_m\big|
\le C_1(\delta) e^{-q(1-\delta)} \|f_m\|_{C(\boldsymbol{\Sigma}_m)}, 
\end{equation}
an error estimate in terms of the number of indices in the anisotropic sparse grid, 
\begin{equation}\label{eq:errestnumberind}
\big|\big({\bf I}-\mathcal{A}_{\bf w}(q,m)\big) f_m\big| 
\le C_2({\bf w}, \beta) \#  X_{\bf w}(q,m)^{-(\beta-1)},
\end{equation}
an estimate of the computational cost,
\begin{equation}\label{eq:finalcompl}
N(q)\isdef\text{cost}\big(\mathcal{A}_{\bf w}(q,m) \big)\le \#  X_{\bf w}(q,m)^2,
\end{equation}
and finally an improved estimate on the number of indices in an anisotropic sparse grid,
\[
 \#  X_{\bf w}(q,m)\le \prod_{n=1}^m \bigg(\frac{ q}{nw_n}+1\bigg). 
\]

With these estimates at hand, we are now able to conduct the main result of this section.
From \eqref{eq:errestnumberind} and \eqref{eq:finalcompl}, we immediately 
obtain that the error in terms of number of quadrature points is bounded 
independently of the dimension \(m\). 
\begin{theorem}\label{theo:finalest}
The error of the anisotropic sparse grid quadrature with \(w_n = \log(\kappa_n)\) 
can be bounded in terms of the number of quadrature points \(N(q)\) according to
\[
\big|\big({\bf I}-\mathcal{A}_{\bf w}(q,m)\big) f_m\big| 
\le c({\bf w}, \beta) N(q)^{-(\beta-1)/2} 
\] 
for all \(\beta < r\). 
\end{theorem}

\section{Numerical results}\label{sec:NumRes}
This section is dedicated to numerical results in order to 
illustrate the theoretical findings. We will consider three different examples:
At first, we consider a high dimensional quadrature problem, afterwards we have a look at
a parametric diffusion problem as they usually occur in the context of uncertainty quantification
and finally, we consider the approximation of quantities of interest from a diffusion problem 
on a random domain.\footnote{The implementation of the sparse grid quadrature is 
available on https://github.com/muchip/SPQR.}

\subsection{Pure quadrature problem}\label{subsec:ex1} As a first example, we consider the quadrature problem
\[
\int_{\Gamma^m}f_m({\bf y})2^{-m}\d{\bf y},
\]
where the function \(f_m\colon\Gamma^m\to\mathbb{R}\) is given by
\[
f_m({\bf y})\isdef\bigg(0.6+0.2\sum_{n=1}^m n^{-s}y_n\bigg)^{-1}\quad\text{for }s=2,3,4.
\]
The derivatives of this function read
\[
\partial_{\bf y}^\balpha f_m = |\balpha|!(-1)^{|\balpha|}\bgamma^\balpha f_m^{|\balpha|+1}\quad\text{with } \gamma_n\isdef0.2n^{-s}.
\]
Therefore, \(f_m\) satisfies Assumption~\ref{ass:AnaExtension} 
for \(\tau_n<5n^{s-1}\). The dimension is set to \(m=1000\) 
and the weight sequence is computed according to
\[
w_n = \log\kappa_n,\quad\text{where }\kappa_n = n^s + \sqrt{1+n^{2s}}.
\]
This means that we disregard the fact that the existence of the 
analytic extension of \(f_m\) can only be proven for 
\(\tau_n<5n^{s-1}\).

A reference solution is obtained by the anisotropic sparse grid 
quadrature on a higher level, featuring about \(10^6\) quadrature 
points. The corresponding values are denoted in Table~\ref{table:refsols}.
\begin{table}[htb]
\begin{tabular}{|c|c|c|}
\hline
\(s = 2\) & \(s = 3\) & \(s = 4\)\\
\hline
1.7393632457035437 & 1.7342253547471955 & 1.7331866232415222\\
\hline
\end{tabular}
\caption{\label{table:refsols}Reference solutions for the different choices of \(s\) in the first example.}
\end{table}
In order to validate these reference solutions, we have tested them against a quasi-Monte Carlo
quadrature based on the Halton sequence.

Next, we consider the convergence of the anisotropic sparse grid quadrature. 
Obviously, in the absence of any decay, a genuine 1000-dimensional problem 
would be computationally not feasible. Thus, in order to determine the 
inherent dimensionality for each choice of the parameter \(s\), we 
approximate the reference solution, which has been computed 
with \(m=1000\), by \(m=10,100,1000\). 

\begin{figure}[htb]
\begin{center}
\pgfplotsset{width=0.47\textwidth, height=0.48\textwidth}
\begin{tikzpicture}
\begin{loglogaxis}[grid, ymin= 0.5e-10, ymax = 1e-0, xmin = 1, xmax =1e6, ytick={0.1,0.01,0.001,0.0001,0.00001,0.000001,1e-7,1e-8,1e-9,1e-10},%
    legend style={legend pos=north east,font=\small}, ylabel={error}, xlabel ={$N$}]
\addplot[line width=0.7pt,color=gray,mark=o] table{./ResultsRev/Res210.txt};\addlegendentry{$m = 10$};
\addplot[line width=0.7pt,color=azure,mark=triangle] table{./ResultsRev/Res2100.txt};\addlegendentry{$m = 100$};
\addplot [line width=0.7pt, color=dunkelmint,mark=square] table{./ResultsRev/Res21000.txt};\addlegendentry{$m  = 1000$};
\addplot[line width=0.7pt, color=black,dashed] table[x index={0},y index={2}, y={create col/linear regression={y}}]{./ResultsRev/Res21000.txt};
\xdef\slope{\pgfplotstableregressiona}
\addlegendentry{$N^{\pgfmathprintnumber{\slope}}$};

\end{loglogaxis}
\end{tikzpicture}
\pgfplotsset{width=0.47\textwidth, height=0.48\textwidth}
\begin{tikzpicture}
\begin{semilogyaxis}[grid, ymin=1,xmin = 0, xmax = 21, xtick={1,5,10,15,20},ytick={10,1e2,1e3,1e4,1e5,1e6},%
    legend style={legend pos=north west,font=\small}, ylabel=points, xlabel ={$q$}]
\addplot[line width=1pt,color=gray,mark=o] table[x index=0,y index=1]{./ResultsRev/pkte2100.txt};\addlegendentry{$N(q)$};
\addplot[line width=1pt,color=azure,mark=triangle] table[x index=0,y index=2]{./ResultsRev/pkte2100.txt};\addlegendentry{\#$X_{\bf w}(q,m)$};
\addplot [line width=1pt, color=dunkelmint,mark=square] table[x index=0,y index=3]{./ResultsRev/pkte2100.txt};
\addlegendentry{SG Formula};
\addplot [line width=1pt, color=blue,mark=diamond] table[x index=0,y index=4]{./ResultsRev/pkte2100.txt};
\addlegendentry{$\log(m)^{q/s}$};
\end{semilogyaxis}
\end{tikzpicture}
\caption{\label{fig:QuadProb2}Convergence of the anisotropic sparse grid 
quadrature with respect to the different parameter dimensions for $s=2$ (left). 
Number of quadrature points, cardinality of the sparse index set and their estimates 
for dimension \(m=100\) (right).}
\end{center}
\end{figure}
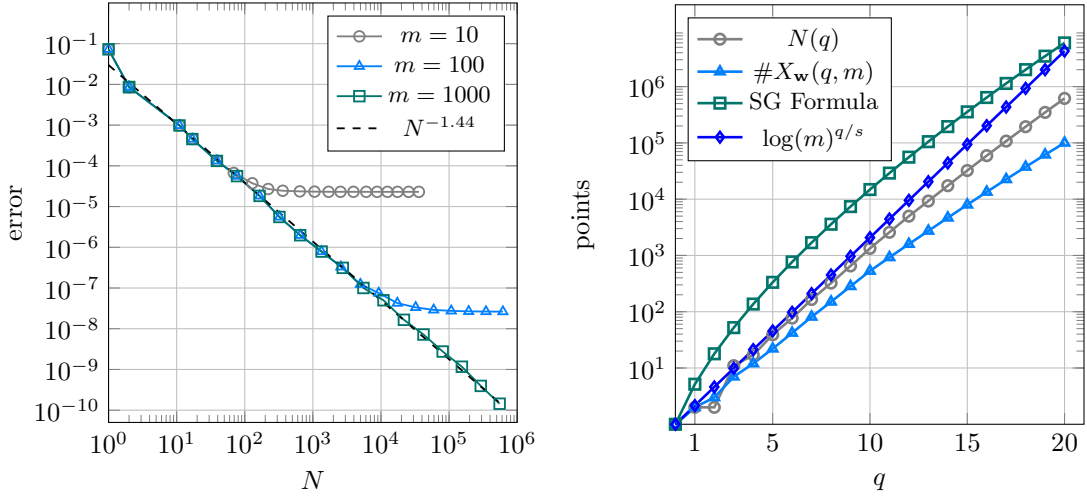

The plot on the left-hand side of Figure~\ref{fig:QuadProb2} shows
the convergence for \(s=2\). It turns out, that we are able to recover 
the reference solution up to an error of order \(10^{-5}\) by just 
considering \(m=10\) dimensions. For \(m=100\), we already achieve 
convergence up to an error of order \(10^{-8}\) and, finally, for \(m=1000\), 
we have convergence up to an error of order \(10^{-10}\). The observed 
convergence rate is in all cases superlinear. Hence, the observed rate is 
much better than expected by Theorem \ref{theo:finalest} since \((\beta-1)/2\) 
would be at most \(1/2\). 

The plot on the right-hand side of 
Figure~\ref{fig:QuadProb2} demonstrates that the bound on the number 
of quadrature points \(N(q)\le \# X_{\bf w}(q,m)^2\) is not sharp here. In fact, 
it seems that the number of quadrature points behaves more linearly in terms 
of the cardinality of the sparse grid index set. Indeed, for all the three different 
settings of this example (i.e., for $s=2,3,4$), the number of quadrature points lies 
between the cardinality of the sparse index set and the novel upper bound 
\eqref{eq:novest}. This is the reason why we observe convergence rates that 
are better reflected by \(\beta-1\) than by \((\beta-1)/2\). However, the novel upper 
bound is a tremendous improvement compared to \eqref{eq:Begdov}: Inserting in the 
weight vector ${\bf w}$ and the considered range of $q$, we obtain values between $9.8\times 10^{43}$ 
and $1.0\times10^{45}$ as an upper estimate for $\#X_{\bf w}(q,m)$ from \eqref{eq:Begdov}. Hence, we refrain 
from including this estimate in Figure~\ref{fig:QuadProb2} in order to illustrate the behavior of the 
remaining formulas more clearly. A graphical comparison of the different upper bounds from 
Section \ref{sec:sharp} is presented in is provided in Figure~\ref{fig:PointEstimates}.
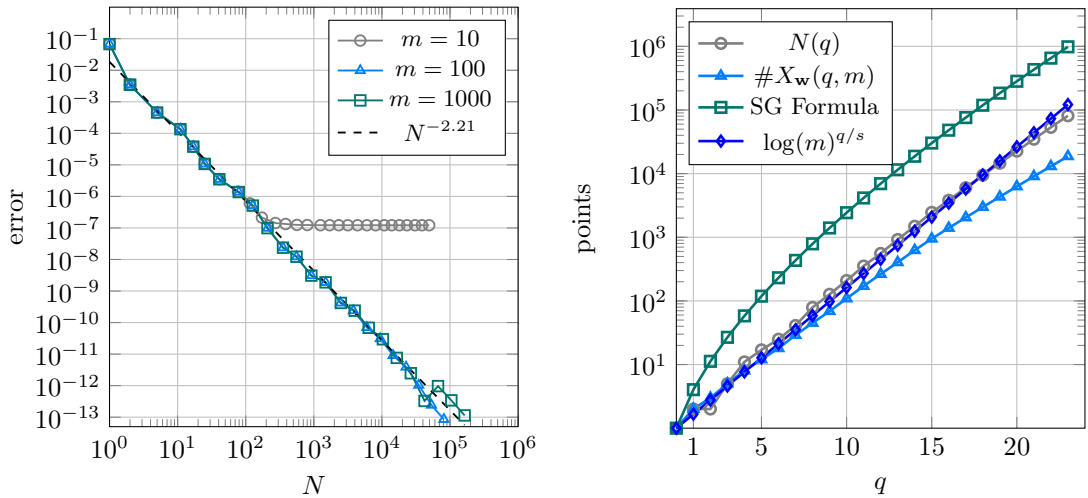
\begin{figure}[htb]
\begin{center}
\pgfplotsset{width=0.47\textwidth, height=0.48\textwidth}
\begin{tikzpicture}
\begin{loglogaxis}[grid, ymin= 0.5e-13, ymax = 1e-0, xmin = 1, xmax =1e6, ytick={0.1,0.01,0.001,0.0001,0.00001,0.000001,1e-7,1e-8,1e-9,1e-10,1e-11,1e-12,1e-13},%
    legend style={legend pos=north east,font=\small}, ylabel={error}, xlabel ={$N$}]
\addplot[line width=0.7pt,color=gray,mark=o] table{./ResultsRev/Res310.txt};\addlegendentry{$m = 10$};
\addplot[line width=0.7pt,color=azure,mark=triangle] table{./ResultsRev/Res3100.txt};\addlegendentry{$m = 100$};
\addplot [line width=0.7pt, color=dunkelmint,mark=square] table{./ResultsRev/Res31000.txt};\addlegendentry{$m  = 1000$};
\addplot[line width=0.7pt, color=black,dashed] table[x index={0},y index={2}, y={create col/linear regression={y}}]{./ResultsRev/Res31000.txt};
\xdef\slope{\pgfplotstableregressiona}
\addlegendentry{$N^{\pgfmathprintnumber{\slope}}$};
\end{loglogaxis}
\end{tikzpicture}
\pgfplotsset{width=0.47\textwidth, height=0.48\textwidth}
\begin{tikzpicture}
\begin{semilogyaxis}[grid, ymin=1,xmin = 0, xmax = 24, xtick={1,5,10,15,20},ytick={10,1e2,1e3,1e4,1e5,1e6},%
    legend style={legend pos=north west,font=\small}, ylabel=points, xlabel ={$q$}]
\addplot[line width=1pt,color=gray,mark=o] table[x index=0,y index=1]{./ResultsRev/pkte3100.txt};\addlegendentry{$N(q)$};
\addplot[line width=1pt,color=azure,mark=triangle] table[x index=0,y index=2]{./ResultsRev/pkte3100.txt};\addlegendentry{\#$X_{\bf w}(q,m)$};
\addplot [line width=1pt, color=dunkelmint,mark=square] table[x index=0,y index=3]{./ResultsRev/pkte3100.txt};
\addlegendentry{SG Formula};
\addplot [line width=1pt, color=blue,mark=diamond] table[x index=0,y index=4]{./ResultsRev/pkte3100.txt};
\addlegendentry{$\log(m)^{q/s}$};
\end{semilogyaxis}
\end{tikzpicture}

\caption{\label{fig:QuadProb3}Convergence of the anisotropic 
sparse grid quadrature with respect to different parameter dimensions for $s=3$ (left). 
Number of points of the quadrature, cardinality of the sparse index set and their estimates 
for dimension \(m=100\) (right).}
\end{center}
\end{figure}

Figure~\ref{fig:QuadProb3} depicts the situation for \(s=3\). Here for \(m=10\), we have 
already convergence up to an error of order \(10^{-7}\). The choices \(m=100,1000\)
both converge towards the reference solution up to an error of order \(10^{-13}\). For all 
choices of \(m\), we observe an order of convergence that is greater than 2. 
The right-hand side of the figure shows that the number of quadrature points \(N(q)\) 
behaves very similar to \(\log(m)^{q/s}\) and that \eqref{eq:novest} is, overestimating 
\(N(q)\) by a factor 10 for most of the considered values of \(q\). The behavior of 
\eqref{eq:Begdov} for $s=3$ is very similar to the case $s=2$ with values ranging 
from $1.7\times 10^{44}$ and $1.1\times 10^{45}$.

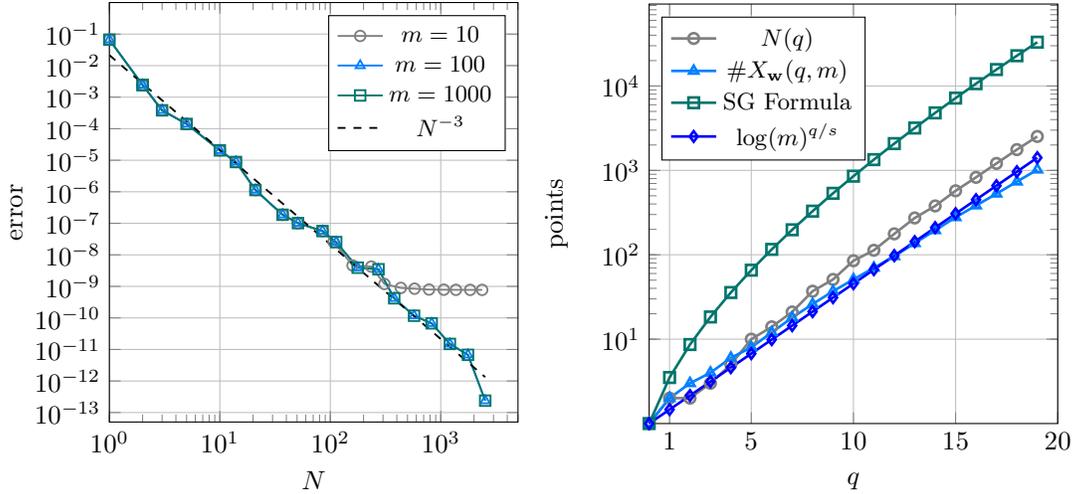
\begin{figure}[htb]
\begin{center}
\pgfplotsset{width=0.47\textwidth, height=0.48\textwidth}
\begin{tikzpicture}
\begin{loglogaxis}[grid, ymin= 0.5e-13, ymax = 1e-0, xmin = 1, xmax =5e3, ytick={0.1,0.01,0.001,0.0001,0.00001,0.000001,1e-7,1e-8,1e-9,1e-10,1e-11,1e-12,1e-13},%
    legend style={legend pos=north east,font=\small}, ylabel={error}, xlabel ={$N$}]
\addplot[line width=0.7pt,color=gray,mark=o] table{./ResultsRev/Res410.txt};\addlegendentry{$m = 10$};
\addplot[line width=0.7pt,color=azure,mark=triangle] table{./ResultsRev/Res4100.txt};\addlegendentry{$m = 100$};
\addplot [line width=0.7pt, color=dunkelmint,mark=square] table{./ResultsRev/Res41000.txt};\addlegendentry{$m  = 1000$};
\addplot[line width=0.7pt, color=black,dashed] table[x index={0},y index={2}, y={create col/linear regression={y}}]{./ResultsRev/Res41000.txt};
\xdef\slope{\pgfplotstableregressiona}
\addlegendentry{$N^{\pgfmathprintnumber{\slope}}$};
\end{loglogaxis}
\end{tikzpicture}
\pgfplotsset{width=0.47\textwidth, height=0.48\textwidth}
\begin{tikzpicture}
\begin{semilogyaxis}[grid, ymin=1,xmin = 0, xmax = 20, xtick={1,5,10,15,20},ytick={10,1e2,1e3,1e4,1e5,1e6},%
    legend style={legend pos=north west,font=\small}, ylabel=points, xlabel ={$q$}]
\addplot[line width=1pt,color=gray,mark=o] table[x index=0,y index=1]{./ResultsRev/pkte4100.txt};\addlegendentry{$N(q)$};
\addplot[line width=1pt,color=azure,mark=triangle] table[x index=0,y index=2]{./ResultsRev/pkte4100.txt};\addlegendentry{\#$X_{\bf w}(q,m)$};
\addplot [line width=1pt, color=dunkelmint,mark=square] table[x index=0,y index=3]{./ResultsRev/pkte4100.txt};
\addlegendentry{SG Formula};
\addplot [line width=1pt, color=blue,mark=diamond] table[x index=0,y index=4]{./ResultsRev/pkte4100.txt};
\addlegendentry{$\log(m)^{q/s}$};
\end{semilogyaxis}
\end{tikzpicture}

\caption{\label{fig:QuadProb4}Convergence of the anisotropic sparse 
grid quadrature with respect to different parameter dimensions for $s=4$ (left). 
Number of points of the quadrature, cardinality of the sparse index set 
and their estimates for dimension \(m=100\) (right).}
\end{center}
\end{figure}

Finally, on the left-hand side of Figure~\ref{fig:QuadProb4}, we see the 
convergence of the sparse grid quadrature for \(s=4\). Here, \(m=10\) already 
provides convergence up to an error of order \(10^{-9}\), whereas \(m=100,1000\) 
converge towards the reference solution up to an error of order \(10^{-13}\).
We observe an order of convergence that is about 3 for all choices of \(m\). 
On the right-hand side of Figure~\ref{fig:QuadProb4}, we see that the number 
of quadrature points \(N(q)\) is bounded by \(2\cdot\# X_{\bf w}(q,m)\) for 
the considered values of \(q\). 
Moreover, it seems that the simplified bound from Lemma \ref{lem:estpointslog}, 
with constant \(c(r)=1\), reflects the behaviour of \(\# X_{\bf w}(q,m)\) quite well.
As for \(s=2,3\), \eqref{eq:Begdov} heavily overestimates the number of 
points for \(s=4\) providing values ranging from 
$2.5\times 10^{44}$ and $8.0\times 10^{44}$.

The example indicates that the proven estimates are by 
far not sharp and could probably be improved further.
\subsection{Random diffusion problem} 

Let \((\Omega,\mathcal{F},\mathbb{P})\) be a complete
probability space and consider the diffusion equation
\begin{equation}\label{eq:eindimprob}
-\partial_x\big(a(x,\omega)\partial_xu(x,\omega)\big)=1\text{ in }D=(0,1)
	\quad\text{for almost every }\omega\in\Omega,
\end{equation}
where the random coefficient $a(x,\omega)$ is uniformly
bounded and elliptic meaning that
\[
0 < \underline{a}\isdef\essinf_{(x,\omega)\in D\times\Omega} a(x,\omega)
	\quad\text{and}\quad \overline{a}\isdef\esssup_{(x,\omega)\in D\times\Omega} a(x,\omega)<\infty.
\]
We arrive at a well-posed problem by complementing \eqref{eq:eindimprob}
with homogenous boundary conditions, i.e.~\(u(0,\omega)=u(1,\omega)=0\).

The first step towards the solution for this class of problems is the parameterization
of the stochastic parameter. To that end, one decomposes the diffusion coefficient
with the aid of the Karhunen-Lo\`eve expansion. Let
the covariance kernel of \(a(x,\omega)\in L^2\big(\Omega;L^2(D)\big)\) be defined by the positive semi-definite function
\[
\mathcal{C}(x,x')\isdef\int_\Omega \big(a({x},\omega)-\E[a](x)\big)\big(a({x'},\omega)-\E[a](x')\big)
\d\mathbb{P}(\omega).
\]
Herein, the integral with respect to \(\Omega\) has to be understood in terms of a Bochner integral, cf.~\cite{HP57}.
Now, let \((\lambda_k,\varphi_k)\) denote the eigenpairs obtained by solving the
eigenproblem for the diffusion coefficient's covariance, i.e.~
\[
\int_0^1\mathcal{C}(x,x')\varphi_k(x')\d x'=\lambda_k\varphi_k(x).
\]
Then, the Karhunen-Lo\`eve expansion of \(a(x,\omega)\) is given by
\[
a({x},\omega)
     =\E[a](x)+\sum_{n=1}^\infty\sqrt{\lambda_n}\varphi_n(x)X_n(\omega),
\]
where \(X_n\colon\Omega\to\Gamma\subset\mathbb{R}\) for \(n=1,2,\ldots\) are centered, pairwise uncorrelated and
\(L^2\)-normalized
random variables with \(X_n\sim\mathcal{U}([-\sqrt{3},\sqrt{3}])\). 
Note that the scaling factor \(\sqrt{3}\)
stems from the \(L^2\)-normalization of the random varibles. 
We have additionally to assume that the random variables
are independent. 

By substituting the random variables with their image, we arrive
in the uniformly distributed case at the parameterized
Karhunen-Lo\`eve expansion
\[
a({x},\boldsymbol\psi)
     = \E[a](x)+\sum_{n=1}^\infty\sqrt{\lambda_n}\varphi_n(x)\sqrt{3}\psi_n,
\]
where \(\psi_n\in\Gamma\). 
We define \(\gamma_n=\sqrt{\lambda_n}\|\varphi_n\|_{L^\infty(D)}\). The decay
of the sequence
\(\{\gamma_n\}_n\) is important in order to determine the region of analytical
extendability of the solution \(u\), cf.~Lemma \ref{lem:analgal}.

Truncating the respective Karhunen-Lo\`eve expansion after \(m\in\mathbb{N}\) terms,
yields the parametric and truncated diffusion problem
\begin{equation}\label{eq:prob1d}
-\partial_x\big(a_m(x,{\bf y})\partial_xu_m(x,{\bf y})\big)=1\text{ in }D=(0,1)\quad\text{for almost every }{\bf y}\in\Gamma^m.
\end{equation}
The impact of truncating the Karhunen-Lo\`eve expansion on the solution
is bounded by
\[
\|u-u_{m}\|_{L^2(\Gamma^\infty;H^1_0(D))}\leq c \|a-a_{m}\|_{L^2(\Gamma^\infty;L^\infty(D))}=\varepsilon(m),\quad c>0,
\]
where \(\varepsilon(m)\to 0\) montonically as \(m\to\infty\), see e.g.~\cite{Cha12,ST06}.
Since the \(L^2\)-norm is stronger than the \(L^1\)-norm, this particularly implies the
approximation estimate \eqref{eq:IntApprox} for \(u\) and \(u_m\), where the modulus
has to be replaced by the \(H^1_0(D)\)-norm.

Given the parametric solution \(u_m(x,{\bf y})\), we are interested in determining properties
of its distribution. In our numerical examples, we focus on the computation of the solution's
moments. These are given by the Bochner integral
\[
\mathcal{M}^p_{u_m}(x)\isdef\int_{\Gamma^m} u_m^p(x,{\bf y})2^{-m}\d{\bf y}.
\]
Especially, there holds \(\E[u_m](x)=\mathcal{M}^1_{u_m}(x)\). 

We remark that the sparse grid quadrature straightforwardly extends 
to Bochner integrals, see e.g.~\cite{HPS13c}. This is due to the fact, 
that the Bochner integral corresponds for almost every \(x\in(0,1)\) 
with the usual Lebesgue integral.

It remains to provide the related regularity results that allow for an analytic extension
of \(u_m\) into the complex plane.
The extendability is guaranteed by \cite[Corollary 2.1]{BNTT12},
which has slightly been modified to fit our purposes.
\begin{lemma}\label{lem:analgal}
The solution \(u_m\) to \eqref{eq:prob1d}
admits an analytic extension into the region
\(\boldsymbol{\Sigma}_m\) for
all \(\btau\) with
\[
\tau_k < \frac{\underline{a}}{C(\delta)k^{1+\delta}{\gamma}_k},\quad
\text{where}\ C({\delta})=\sum_{k=1}^\infty k^{-1-\delta}\text{ for
arbitrary \(\delta>0\).}
\]
Moreover, it holds that
\(
\|u_m\|_{C(\boldsymbol{\Sigma}_m;H_0^1(D))}
\) 
is bounded. 
\end{lemma}

Lemma \ref{lem:analgal} characterizes the region of analyticity and, therefore, 
Assumption~\ref{ass:AnaExtension} for this application. 
The next lemma from \cite{BNT07} guarantees that \(u_m\colon\Gamma^m\to H_0^1(D)\) also satisfies
a one-dimensional error estimate for the polynomial approximation which gives rise to an
estimate that is similar to \eqref{eq:assgaussquad}.

\begin{lemma}\label{lem:uniquaderr}
Let \(\mathcal{X}\) be a Banach space. Suppose that \(v\in C(\Gamma;\mathcal{X})\) admits
an analytic extension into \(\Sigma_1\).
Then, the error of the best approximation by polynomials of degree
at most \(n\) can be bounded by
\begin{equation}\label{eq:onedimerruni}
\inf_{w\in\Pi_{n}\otimes \mathcal{X}}
\|v-w\|_{C(\Gamma;\mathcal{X})}\le
\frac{2}{\kappa_1-1} e^{-n \log(\kappa_1)}
\|v\|_{C(\Sigma_1;\mathcal{X})}
\end{equation}
with \(\kappa_1=\tau_1+\sqrt{1+\tau_1^2}\).
\end{lemma}
Thus, in view of \eqref{eq:polappr}, we end up with the estimate 
\begin{equation*}
\big\|(\intI^{(n)}-\intQ^{(n)}) u_m({\bf y}_n^\star)\big\|_{H^1_0(D)}
\leq c(\kappa_n) \exp\big(-\log(\kappa_n) (2N-1)\big)\|u_m({\bf y}_n^\star)\|_{C(\Sigma_n;H^1_0(D)}. 
\end{equation*}
More generally we remark that by simply replacing all absolute values by the norms 
of the underlying Banach space \(\mathcal{X}\), 
the whole analysis we have presented so far for real valued functions directly transfers 
to Banach space valued functions \(f_m\colon\Gamma^m\to\mathcal{X}\) which fulfill the estimate
\begin{equation*}
\big\|(\intI^{(n)}-\intQ^{(n)}) f_m({\bf y}_n^\star)\big\|_{\mathcal{X}}
\leq c(\kappa_n) \exp\big(-\log(\kappa_n) (2N-1)\big)\|f_m({\bf y}_n^\star)\|_{C(\Sigma_n;\mathcal{X})}. 
\end{equation*}
\begin{figure}[htb]
\begin{center}
\pgfplotsset{width=0.49\textwidth, height=0.49\textwidth}
\begin{tikzpicture}
\begin{loglogaxis}[grid, ymin= 1e-8, ymax = 1e-0, xmin = 1, xmax =2e5, ytick={0.1,0.01,0.001,0.0001,0.00001,0.000001,1e-7,1e-8,1e-9,1e-10,1e-11,1e-12,1e-13},%
    legend style={legend pos=north east,font=\small}, ylabel={\(H^1\)-error}, xlabel ={$N$}]
\addplot[line width=0.7pt,color=gray,mark=o] table[x index=0,y index=1]{./ResultsRev/convergence52.txt};\addlegendentry{$\mathbb{E}_u$};
\addplot[line width=0.7pt,color=azure,mark=triangle] table[x index=0,y index=2]{./ResultsRev/convergence52.txt};\addlegendentry{$\mathcal{M}_u^2$};
\addplot [line width=0.7pt, color=dunkelmint,mark=square] table[x index=0,y index=3]{./ResultsRev/convergence52.txt};\addlegendentry{$\mathcal{M}_u^3$};
\addplot [line width=0.7pt, color=blue,mark=diamond] table[x index=0,y index=4]{./ResultsRev/convergence52.txt};\addlegendentry{$\mathcal{M}_u^4$};
\addplot [line width=0.7pt, color=black,dashed] table[x index={0}, y expr={0.8*x^(-1.13)}]{./ResultsRev/convergence52.txt};
\addplot [line width=0.7pt, color=black,dashed] table[x index={0}, y expr={0.008*x^(-1.13)}]{./ResultsRev/convergence52.txt};\addlegendentry{$N^{-1.13}$};
\end{loglogaxis}
\end{tikzpicture}
\pgfplotsset{width=0.49\textwidth, height=0.49\textwidth}
\begin{tikzpicture}
\begin{loglogaxis}[grid, ymin= 1e-9, ymax = 1e-0, xmin = 1, xmax =2e5, ytick={0.1,0.01,0.001,0.0001,0.00001,0.000001,1e-7,1e-8,1e-9,1e-10,1e-11,1e-12,1e-13},%
    legend style={legend pos=north east,font=\small}, ylabel={\(H^1\)-error}, xlabel ={$N$}]
\addplot[line width=0.7pt,color=gray,mark=o] table[x index=0,y index=1]{./ResultsRev/convergence72.txt};\addlegendentry{$\mathbb{E}_u$};
\addplot[line width=0.7pt,color=azure,mark=triangle] table[x index=0,y index=2]{./ResultsRev/convergence72.txt};\addlegendentry{$\mathcal{M}_u^2$};
\addplot [line width=0.7pt, color=dunkelmint,mark=square] table[x index=0,y index=3]{./ResultsRev/convergence72.txt};\addlegendentry{$\mathcal{M}_u^3$};
\addplot [line width=0.7pt, color=blue,mark=diamond] table[x index=0,y index=4]{./ResultsRev/convergence72.txt};\addlegendentry{$\mathcal{M}_u^4$};
\addplot [line width=0.7pt, color=black,dashed] table[x index={0}, y expr={1*x^(-1.29)}]{./ResultsRev/convergence52.txt};
\addplot [line width=0.7pt, color=black,dashed] table[x index={0}, y expr={0.008*x^(-1.29)}]{./ResultsRev/convergence52.txt};\addlegendentry{$N^{-1.29}$};
\end{loglogaxis}
\end{tikzpicture}
\caption{\label{fig:Mat72conv}Convergence of the anisotropic sparse grid 
quadrature for the random diffusion problem (\(\nu=5/2\) left, \(\nu=7/2\) right).}
\end{center}
\end{figure}
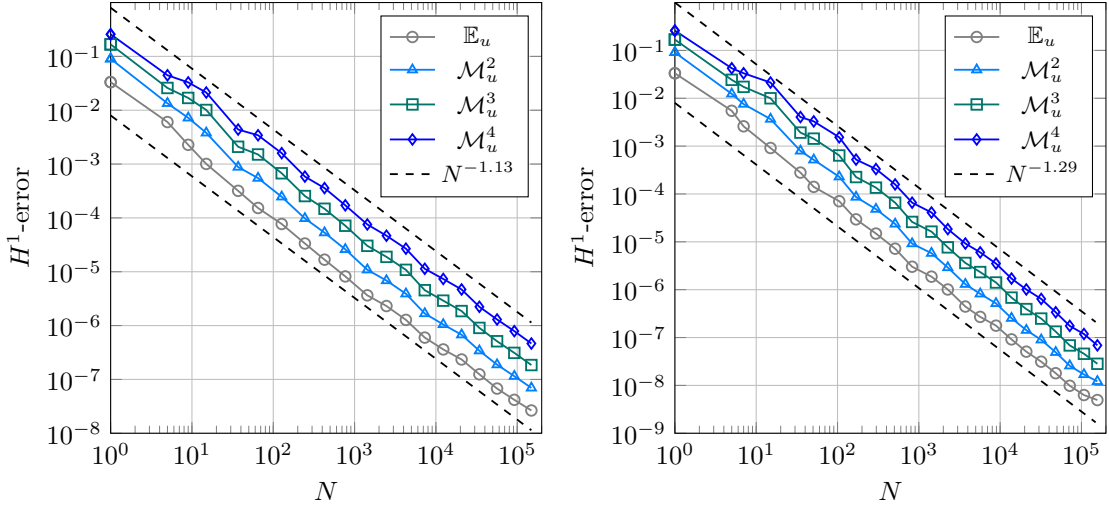
\begin{figure}[htb]
\begin{center}
\pgfplotsset{width=0.47\textwidth, height=0.48\textwidth}
\begin{tikzpicture}
\begin{semilogyaxis}[grid, ymin=1,xmin = 1, xmax = 40, xtick={1,10,20,30,40},ytick={10,1e10,1e20,1e30,1e40,1e40,1e50},%
    legend style={legend pos=north west,font=\small}, ylabel=points, xlabel ={$q$}]
\addplot[line width=1pt,color=gray,mark=o] table[x index=0,y index=1]{./ResultsRev/pointsMat52.txt};\addlegendentry{\#$X_{\bf w}(q,m)$};
\addplot[line width=1pt,color=azure,mark=triangle] table[x index=0,y index=2]{./ResultsRev/pointsMat52.txt};\addlegendentry{SG Formula};
\addplot [line width=1pt, color=dunkelmint,mark=square] table[x index=0,y index=3]{./ResultsRev/pointsMat52.txt};\addlegendentry{BD Formula};
\addplot [line width=1pt, color=blue,mark=diamond] table[x index=0,y index=4]{./ResultsRev/pointsMat52.txt};\addlegendentry{TP Formula};
\end{semilogyaxis}
\end{tikzpicture}
\pgfplotsset{width=0.47\textwidth, height=0.48\textwidth}
\begin{tikzpicture}
\begin{semilogyaxis}[grid, ymin=1,xmin = 1, xmax = 40, xtick={1,10,20,30,40}, ytick={10,1e5,1e10,1e15,1e20,1e25,1e30},%
    legend style={legend pos=north west,font=\small}, ylabel=points, xlabel ={$q$}]
\addplot[line width=1pt,color=gray,mark=o] table[x index=0,y index=1]{./ResultsRev/pointsMat72.txt};\addlegendentry{\#$X_{\bf w}(q,m)$};
\addplot[line width=1pt,color=azure,mark=triangle] table[x index=0,y index=2]{./ResultsRev/pointsMat72.txt};\addlegendentry{SG Formula};
\addplot [line width=1pt, color=dunkelmint,mark=square] table[x index=0,y index=3]{./ResultsRev/pointsMat72.txt};\addlegendentry{BD Formula};
\addplot [line width=1pt, color=blue,mark=diamond] table[x index=0,y index=4]{./ResultsRev/pointsMat72.txt};\addlegendentry{TP Formula};
\end{semilogyaxis}
\end{tikzpicture}
\caption{\label{fig:PointEstimates}Estimated number of points from the different formulas (\(\nu=5/2\) left, \(\nu=7/2\) right).}
\end{center}
\end{figure}
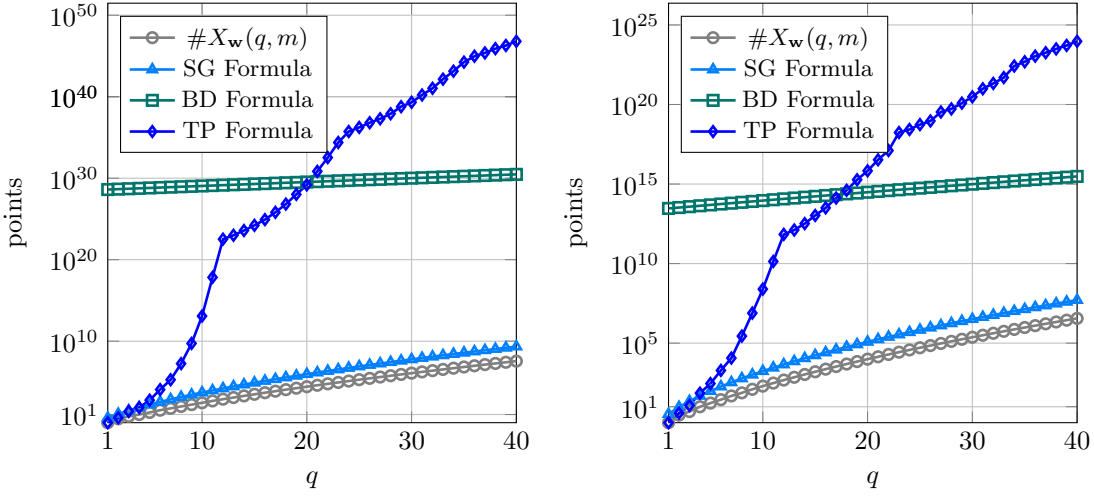

For this example, we employ two covariance kernels of
the Mat\'ern class for \(\nu=5/2\) and \(\nu=7/2\), cf.~\cite{Mat86}, i.e.
\[
\mathcal{C}_{5/2}(\rho)\isdef\frac 1 4\bigg(1 +\frac{\sqrt{5}\rho}{\ell} + \frac{5\rho^2}{3\ell^2}\bigg)\exp\bigg(-\frac{\sqrt{5}\rho}{\ell}\bigg)
\]
and
\[
\mathcal{C}_{7/2}(\rho)\isdef\frac 1 4\bigg(1 +\frac{\sqrt{7}\rho}{\ell} + \frac{14\rho^2}{5\ell^2}+\frac{49\sqrt{7}\rho^3}{15\ell^2}\bigg)\exp\bigg(-\frac{\sqrt{7}\rho}{\ell}\bigg),
\]
where \(\rho\isdef|x-x'|\). The correlation length is in both cases set to \(\ell=1/2\).
The spatial discretization is performed with piecewise linear finite elements an
a mesh with mesh size \(h=2^{-14}\), which results from \(16384\) equidistant
sub-intervals and is good a tradeoff between accuracy and computational time. A numerical approximation to the
Karhunen-Lo{\`e}ve expansion is computed by the pivoted Cholesky
decomposition of the covariance operator with a trace error of \(\varepsilon=2^{-28}\).
This yields an approximation error of the underlying random
field of \(\varepsilon=2^{-14}\), see \cite{HPS14a} for the details. Thus, the finite element error and the 
truncation error for the Karhunen-Lo\`eve expansion are balanced. The related truncation 
rank is given by  \(m=64\) for \(\mathcal{C}_{5/2}\) and  \(m=30\) for \(\mathcal{C}_{7/2}\). 
In addition, we set \(\E[a](x)=2.5\).
From \cite[Corollary 5]{GKN+13}, we know that \(\gamma_n\le c n^{-3}\) for \(\mathcal{C}_{5/2}\) and
\(\gamma_n\le c n^{-4}\) for \(\mathcal{C}_{7/2}\).

Since the solution of \eqref{eq:eindimprob} is not known analytically, we have again to
provide a reference solution. It is computed with \(h=\varepsilon=2^{-14}\), 
as well. Hence, we take here only the quadrature error into account.
The quadrature error is measured relative to the norm of the reference solution.
This reference solution is computed by the quasi-Monte Carlo quadrature
with Halton points and \(N=2^{30}\approx 10^9\) samples.

For the anisotropic sparse grid quadrature, we choose the weights \(w_n\) according 
to \(w_n=\log(\kappa_n)\) with \(\tau_n=1/\gamma_n\). This would be the correct 
setting for a corresponding anisotropic tensor product quadrature.
Hence, our anisotropic sparse grid quadrature is essentially
a sparsification of the anisotropic tensor product quadrature, cf.~\cite{HPS13d}
for more details on the anisotropic tensor product quadrature.
To choose the same quantity \(\tau_n\) for the region of analyticity as for the tensor
product quadrature seems to be a violation of Lemma \ref{lem:analgal}.
Indeed, the assertion of this lemma is that the quantities \(\tau_n\), which describe
the region of analytic extendability in each direction \(\Sigma_n\),
should be rescaled to \(\tilde{\tau}_n=\tau_n/(c(\delta)n^{1+\delta})\) in order to ensure
analytic extendability into the tensor product domain \(\boldsymbol{\Sigma}(\tilde{\btau})\).
Nonetheless, our experience suggests that the sparsification of the
anisotropic tensor product quadrature yields an error which is nearly as good
as the error of the anisotropic tensor product quadrature itself.

For both cases \(\nu=5/2\) and \(\nu=7/2\), it turns out that we obtain similar convergence rates for all moments
up to order four. The smoothness of the underlying covariance kernel has only little influence
on the rate of convergence here. This might be caused by the fact that the eigenvalues in 
the Karhunen-Lo\`eve expansion do not strictly decay in contrast to the previous example, but 
have some offset before the asymptotical rate is achieved. This is due to the small correlation 
length \(\ell=1/2\).

In addition to the convergence studies for the sparse anisotropic quadrature,
we also provide results on the estimated number of indices contained in the
sparse index set. To that end, we compare the tensor product estimate \eqref{eq:anisotpest}
and the formula by Beged-Dov \eqref{eq:Begdov} with
the novel estimate proposed in this article \eqref{eq:novest}. 
It turns that the new estimate fits the cardinality of the sparse grid index set 
very well in comparison to the other estimates, cp.\ Figure~\ref{fig:PointEstimates}.

\subsection{Laplace equation on a random domain}
\begin{figure}[htb]
\begin{center}
\includegraphics[scale=.5, clip,trim= 50 40 35 30]{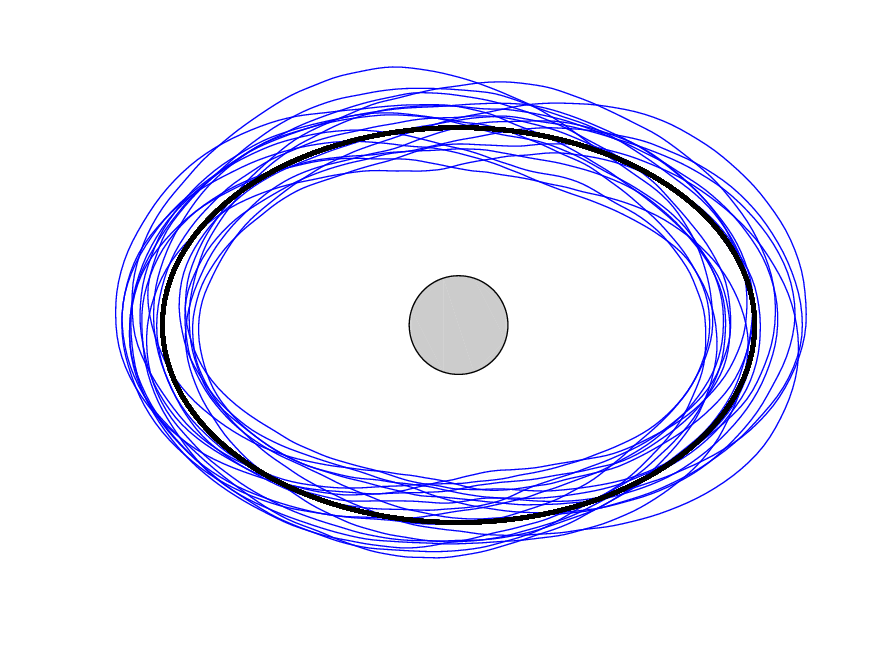}
\caption{\label{Fig:randDom}Different realizations of the random domain.}
\end{center}
\end{figure}
As our third example, we consider the Laplacian on a random domain:
\[
\Delta u(\omega) = 0\ \text{in }D(\omega)\subset\mathbb{R}^2,\quad
 u(\omega) = \frac 1 4 (x_1^2+x_2^2)\ \text{on }\partial D(\omega).
\]
The nominal domain \(D_\refd\isdef\E[D]\) is given by an ellipse with 
semi-axis 0.3 and 0.2. The boundary perturbation is defined by means of 
a random vector field with \({\bf V}(\partial D_\refd,\omega)=\partial D(\omega)\).
This vector field is represented via a Karhunen-Lo\`eve expansion according to
\[
{\bf V}({\bf x},\omega)={\bf x}+\bigg(\frac{1}{40}\sum_{k=1}^{50} k^{-2}\big(\cos(k\phi_{\bf x})X_{2k}(\omega)+\sin(k\phi_{\bf x})X_{2k-1}(\omega)\big)\bigg)\begin{bmatrix}
\cos(\phi_{\bf x})\\ \sin(\phi_{\bf x})\end{bmatrix},
\]
where \(\phi_{\bf x}\in[0,2\pi)\) is the angle associated to \({\bf x}\in\partial D_\refd\) and
\(X_n\sim\mathcal{U}(-\sqrt{3},\sqrt{3})\). 
A visualization of several realizations of the random domain can be found in Figure~\ref{Fig:randDom}.
By extending the boundary perturbation into space by a smooth blending function, it can be shown
that the solution \(u\) to the diffusion problem on the random domain exhibits a regularity similar
to Lemma~\ref{lem:analgal} with \(\gamma_n\isdef\sqrt{\lambda_n}\|\bphi_n\|_{W^{1,\infty}(D_\refd;\mathbb{R}^2)}\),
where the \(\bphi_n\) correspond to the spatial functions of \({\bf V}\)'s Karhunen-Lo\`eve expansion, see \cite{HPS16} for the details. 

As quantity of interest, we consider, likewise to the previous example, 
the first four moments of the solution, but this time restricted to the disc 
\(\{{\bf x}\in\mathbb{R}^2:\|{\bf x}\|_2\leq0.05\}\). This disc is illustrated as
grey shaded area in Figure~\ref{Fig:randDom}. The numerical solution of 
the underlying boundary value problem is performed by an exponentially 
convergent boundary integral method based on collocation as proposed in
\cite{Kre92}. The boundary is discretized by 200 points and a 
reference is again obtained by \(2^{30}\) quasi-Monte Carlo samples 
based on the Halton sequence. 

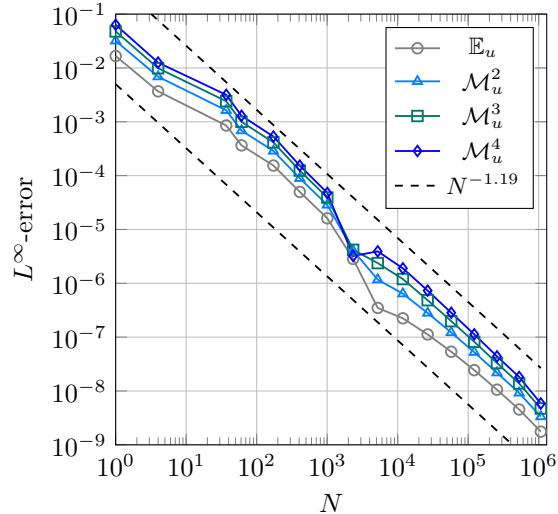
\begin{figure}[htb]
\begin{center}
\pgfplotsset{width=0.49\textwidth, height=0.49\textwidth}
\begin{tikzpicture}
\begin{loglogaxis}[grid, ymin= 1e-9, ymax = 1e-1, xmin = 1, xmax =1.3e6, ytick={0.1,0.01,0.001,0.0001,0.00001,0.000001,1e-7,1e-8,1e-9,1e-10,1e-11,1e-12,1e-13},%
    legend style={legend pos=north east,font=\small}, ylabel={\(L^\infty\)-error}, xlabel ={$N$}]
\addplot[line width=0.7pt,color=gray,mark=o] table[x index=0,y index=1]{./ResultsRev/ausgSG.txt};\addlegendentry{$\mathbb{E}_u$};
\addplot[line width=0.7pt,color=azure,mark=triangle] table[x index=0,y index=2]{./ResultsRev/ausgSG.txt};\addlegendentry{$\mathcal{M}_u^2$};
\addplot [line width=0.7pt, color=dunkelmint,mark=square] table[x index=0,y index=3]{./ResultsRev/ausgSG.txt};\addlegendentry{$\mathcal{M}_u^3$};
\addplot [line width=0.7pt, color=blue,mark=diamond] table[x index=0,y index=4]{./ResultsRev/ausgSG.txt};\addlegendentry{$\mathcal{M}_u^4$};
\addplot [line width=0.7pt, color=black,dashed] table[x index={0}, y expr={0.4*x^(-1.19)}]{./ResultsRev/ausgSG.txt};
\addplot [line width=0.7pt, color=black,dashed] table[x index={0}, y expr={0.005*x^(-1.19)}]{./ResultsRev/ausgSG.txt};\addlegendentry{$N^{-1.19}$};
\end{loglogaxis}
\end{tikzpicture}
\caption{\label{fig:DomainConvergence}Convergence of the anisotropic 
sparse grid quadrature for the random domain.}
\end{center}
\end{figure}

In Figure~\ref{fig:DomainConvergence}, the relative errors of 
the moments computed by the sparse grid quadrature are depicted.
Also in this application, we observe a rate of convergence which 
is greater than 1 for all moments under consideration and, therefore,
again better than expected.

\section{Conclusion}\label{sec:conclusion}

In the present article, we have considered the anisotropic sparse grid quadrature 
applied to a class of analytic functions. Under the assumption that the regions of 
analyticity for the particular dimensions are increasing algebraically, i.e.~\(\tau_n\geq c n^{r}\) 
for some  \(c,r\in\mathbb{R}_+\), we have derived dimension independent convergence 
of the anisotropic sparse grid quadrature. 
In order to estimate the cost of the quadrature, a novel estimate for the 
cardinality of the anisotropic sparse grid index set has been proven. 
Our numerical comparisons suggest that this novel estimate is much more accurate 
than the widely used estimate by Beged-Dov. 
In particular, it has been shown that the estimate of Beged-Dov can easily be
deduced from our novel estimate. 
Besides pure quadrature problems, the anisotropic sparse grid quadrature has been 
successfully applied to Bochner integrals which stem from the computation of the 
moments of elliptic diffusion problems with random coefficients or on random 
domains, respectively. 

\bibliographystyle{plain}

\end{document}